\documentclass[11pt]{article}
\usepackage[english]{babel}

\usepackage{amsmath}
\usepackage{amsthm}
\usepackage{amsfonts,amssymb,latexsym,graphicx,psfrag}
\usepackage{version}
\usepackage{verbatim}
\usepackage{mathrsfs} 
\newtheorem{thm}{Theorem}[section]

\newtheorem{lem}[thm]{Lemma}

\theoremstyle{definition}

\theoremstyle{remark}

\numberwithin{equation}{section}

\def \<{\langle}
\def \>{\rangle}

\begin{document}
\title{Necessary conditions for adverse control problems expressed by relaxed derivatives}
\author{M. Palladino \thanks{Department of Mathematics, Penn State University, University Park, 16802, PA, USA {\tt\small mup26@psu.edu}}}
\maketitle

\begin{abstract}

This paper provides a framework for deriving a new set of necessary conditions for adverse control problems among two players. The distinguish feature of such problems is that the first player has a priori knowledge on the second player strategy. A subclass of adverse control problems is the one of minimax control problems, which frequently arise in robust dynamic optimization. The conditions derived in this manuscript are expressed in terms of relaxed derivatives \cite{key-6}: the dual variables and the related functions are limits of computable sequences, obtained by considering a regularized version of the original problem and applying well known necessary condition \cite{key-4}. This topic was initially treated by J. Warga in \cite{key-7}.
\end{abstract}
{\small 

\noindent

}


\section{Introduction}
\label{intro}
In this paper, we consider adverse control problems between two players
described by differential equations
\[
y(t)=b+\intop_{t_{0}}^{t}f(s,y(s),u(s))ds
\]
\[
\tilde{y}(t)=\tilde{b}+\intop_{t_{0}}^{t}\tilde{f}(s,\hat{y}(s),u(s),v(s))ds.
\]
In this model, $u(.),$ $b,$ and $y(.)$ are, respectively, the control,
the initial state and the state trajectory of the first player, while
$v(.)$, $\tilde{b}$ and $\tilde{y}(.)$ are, respectively, the control,
the initial state and the state trajectory of the second player; $\hat{y}(.)=(y(.),\tilde{y}(.))$
is the state trajectory of the control system and $\hat{b}=(b,\tilde{b})$ is its initial condition at time $t_{0}$. Adverse control problems
concern the choice of $(u(.),b,\tilde{b})$ that minimizes a given function
$h_{0}(u,b)$ and satisfies constraints expressed in terms of 
\[
h_{1}(u,b)=0,\qquad\hat{h}(u,v,b,\tilde{b})\leq0
\]
for every admissible strategy $v(.).$

Adverse control problems have some formal analogy with differential
games, but yield a priori information to the second player about the
first player strategy. The model formulation that we deal with emphasizes this different feature by
 decoupling the player one trajectory $y(.)$ from the player
two trajectory $\tilde{y}(.),$ instead of considering a joint differential
equation 
\[
\hat{y}(t)=\hat{b}+\intop_{t_{0}}^{t}\hat{f}(s,\hat{y}(s),u(s),v(s))ds,
\]
as it is usually taken into account in the differential game framework. A
particular case of adverse control problems are minimax problems,
in which 
\[
h_{0}(u,b)=\sup_{v}\hat{h}_{0}(u,b,v)
\]
and the minimization process is carried out following some ``worst
case'' criteria. A deeper exposition on the topic is presented in the monograph \cite{BoPo}.

Adverse control problems were extensively studied by Warga in his
monograph \cite{key-4}, in which he proposed two extensions of the
original problem, aimed to guarantee the existence of a solution.
He denotes such enlarged problems as relaxed and hyperrelaxed, respectively.
The relaxed extension can be properly used to model the case in which
the function $\tilde{f}$ is additively coupled with respect to the
control strategies of players one and two respectively, that is:
\[
\tilde{f}(t,\hat{y},u,v)=\tilde{f}_{1}(t,\hat{y},u)+\tilde{f}_{2}(t,\hat{y},v);
\]
another case in which the relaxed extension can be successfully applied
is when the second player does not have perfect means to detect the
value $u(t),$ but can just detect an average of values of $u(.)$
over short intervals of time. In all the other cases (which means,
when $\tilde{f}$ assumes a general form and when the second player
can acquire information on the value $u(t)$), the relaxed problem
can fail to provide the ``right'' value of the adverse control problem:
in other words, it can occur that the value of the relaxed extension
is lower than the value of the original problem, even for smooth
dynamics (as it is showed in \cite{key-4}). This lack of properness
motivates the attention for the hyperrelaxed extension: in this setting,
the second player gains more freedom in the choice of the control
strategy, making the hyperrelaxed problem formulation ``fair'', in the sense that the
value of the hyperrelaxed extension does not change with respect to
the original one. In \cite{key-4} and \cite{key-7}, Warga proves the properness of
the hyperrelaxed problem and the existence of a sequence of original
controls which approximates the hyperrelaxed problem solution. 

	In the same monograph, necessary conditions both for relaxed and
hyperrelaxed problems, are derived in the case of smooth data. The nonsmooth setting is 
considered in \cite{key-5}, where necessary conditions are derived using the notion of Warga derivative
container (see \cite{key-9}, \cite{key-8}). However, as it is explained in (\cite{key-5}, section 3), there are some technical difficulties (related to the measurability of the relaxed and hyperrelaxed hamiltonians) which
prevent to obtain a ``pointwise" maximum principle strong as much as in the smooth setting.\\
A different approach to the particular case in which minimax problems are considered is provided in \cite{key-3}, where the set of ``adverse trajectories" is identified with a compact metric space. Necessary conditions are expressed in terms of a nonsmooth Pontryagin maximum principle in which the adjoint equation and the transversality condition are modified in order to gain good compactness properties in the proofs.\\
\\
\noindent
The aim of this paper is to provide a new set of necessary conditions for adverse control problems 
which have a stronger resemblance with the necessary conditions established in \cite{key-4} for the smooth setting. Indeed, the main results, collected in theorems \ref{thm:Hyper-Relaxed} and \ref{thm:Relaxed Problem}, provide a form of Pontryagin maximum principle in which the pointwise maximum condition is still preserved and the adjoint equations have a limit representations. \\
 The key idea of the proofs is to define a sequence of perturbed smooth problems for which well-known necessary conditions \cite{key-4} apply and the couples solutions/multipliers converge to a couple solution/multiplier of the hyper relaxed (or relaxed) problem. 
We do not make use of variational principles, but we regularize the data by convolution techniques. 
We do not provide direct substitute to the non-existing derivative of nonsmooth data, but we cope with the convolution integrals, using the concept of ``relaxed derivative" established in \cite{key-6}.\\
We stress that the sequence solution/multipliers generated in the main proofs could be computed in many cases of interest and no a priori information on the minimizers of the adverse control problem is required. \\

The paper is organized as follows: in sections 2-5 we describe notations, a precise statement of the problem, an overview on relaxation and hyperrelaxation schemes and the assumptions that we refer to through all the paper; in sections 6-8 we provide some tools, lemmas and convergence results for Fredholm approximations and relaxed derivatives; finally, in section 9, the main theorems and their proofs can be found.

\section{Preliminaries and Notations}
\label{sec:1}
In this section, we introduce notations and basic concepts which we
will use through all the paper.

Given a compact set $K,$ we denote as $C(K)$ the set of all continuous
functions defined on $K.$ It is well known that the set $C^{*}(K)$
(the set of linear and continuous functionals defined on $C(K)$)
can be identified with the set of finite Radon measures on $K$,
which we denote as $\mathrm{f.r.m.}(K).$ Further we denote as $\mathrm{f.r.m.}^{+}(K)$
the set of finite positive Radon measures and by $\mathrm{r.p.m.}(K)$
the set of  Radon probability measures. We denote also as $\mathcal{B}(K)$
the Borel $\sigma-$field on $K$ and for every $\mu\in C^{*}(K),$
we denote as $\mu(K)$ the norm in total variation of $\mu.$

Given a measure space $(X,\mu,\mathcal{F})$ and a metric space $(Y,d),$
we denote as $L^{1}(X,\mu,Y):=\mathcal{L}(X,\mu,Y)/\sim,$ where $\mathcal{L}(X,\mu,Y)$
is the set of the $\mu$-integrable functions $f:X\rightarrow Y$ defined at every point $x\in X$
such that $\intop|f(x)|\mu(dx)<\infty$ and $\sim$ is the equivalence
ralation $f\sim g$ iff $f=g$ $\mu$-a.e. In the paper, we use the notation $L^{1}(X,\mu)$ or $L^{1}(\mu,Y)$ when there is no disambiguation in the codomain or the domain, respectively.

Given a set $A$, we denote as $\mathrm{co}\, A$ the convex hull of $A$. Finally, we denote as $\mathbb{B}$ the closed unit ball in the euclidean space with suitable dimension, as $\mathscr{P}(K)$ the power set of $K$ and as $M_{r\times k}$ the set of matrices with $r$ rows and $k$ columns.

\section{Original Problem Statement}
\label{sec:Problem-Statement}

Consider the adverse control problem
\[
(OP)\,\left\{ \begin{array}{l}
\mathrm{Minimize}_{u\in \mathcal{U}} \,h_{0}(y(u)(t_{1}))\\
\mathrm{over\: measurable\: functions\:}u(.), v(.)\mathrm{\: such\: that}\\
u(t)\in U(t ),\quad v(t)\in V(t)\qquad\mathrm{a.e.}\; t\in[t_{0},t_{1}]\\
\mathrm{such\: that,\: for\: each}\:v(.),\\
\dot{y}(t)=f(t,y(t),u(t))\quad\mathrm{a.e.}\: t\in[t_{0},t_{1}]\\
\dot{\tilde{y}}(t)=\tilde{f}(t,\hat{y}(t),u(t),v(t)) \quad \mathrm{a.e.}\: t\in[t_{0},t_{1}]\\
y(t_{0})=b\in B,\quad \tilde{y}(t_{0})=\tilde{b}\in \tilde{B}\quad\mathrm{and}\quad h_{1}(y(u)(t_{1}))=0\\
\hat{h}(y(u,v)(t_{1}))\leq 0 \quad \mathrm{ for\: each}\:v\in \mathcal{V}
\end{array}\right. ,
\]
where $f:[t_{0},t_{1}]\times\mathbb{R}^{n}\times U\rightarrow\mathbb{R}^{n},$
$\tilde{f}:[t_{0},t_{1}]\times\mathbb{R}^{n+m}\times U\times V\rightarrow\mathbb{R}^{m}$, $h_{0}:\mathbb{R}^{n}\rightarrow\mathbb{R},$
$h_{1}:\mathbb{R}^{n}\rightarrow\mathbb{R}$ and $\hat{h}:\mathbb{R}^{n}\times\mathbb{R}^{m}\rightarrow\mathbb{R}$
are given functions, $U$ and $V$ compact metric spaces and $[t_{0}, t_{1}]$ a given interval. The
initial condition $\hat{b}:=(b,\tilde{b})$ takes values on the compact and convex
set $\hat{B}:=B\times \tilde{B}\subset\mathbb{R}^{n}\times \mathbb{R}^{m}$. It turns out that the initial condition can be regarded as a choice of control paremeters for problem $(OP)$ (cfr. \cite{key-4}). We denote as $\hat{y}=(y,\tilde{y})$, as $\hat{f}=(f,\tilde{f})$
and we sometimes emphasize the dependence on the controls writing $y(u)(t),$
$\hat{y}(u,v)(t)$. The mappings $U(.):[t_{0},t_{1}]\rightarrow\mathscr{P}(U)$
and $V(.):[t_{0},t_{1}]\rightarrow\mathscr{P}(V)$ are given Borel measurable multifunctions with
compact values and we denote as $\mathcal{U}$ (risp. $\mathcal{V}$)
the set of all measurable functions $u(.):[t_{0},t_{1}]\rightarrow U$ such that $u(t) \in U$ a.e. $t\in [t_{0},t_{1}]$.
(risp. $v(.):[t_{0},t_{1}]\rightarrow V$ such that $v(t) \in V(t)$ a.e. $t\in [t_{0},t_{1}]$).

\section{Assumptions}
In this paper, we assume the following assumptions on the data: let
$\Omega\subset\mathbb{R}^{n},$ $\tilde{\Omega}\subset\mathbb{R}^{m}$ be open
set, $\hat{\Omega}:=\Omega\times\tilde{\Omega}.$ We consider functions
\[
\hat{f}=(f,\tilde{f}):[t_{0},t_{1}]\times\hat{\Omega}\times U\times V\rightarrow\mathbb{R}^{n+m},\quad h_{i}:\Omega\rightarrow\mathbb{R},\; i=0,1,
\]
\[
\hat{h}:\hat{\Omega}\rightarrow\mathbb{R}
\]
such that 

\begin{itemize}

\item[$H1)$] $\hat{f}(.,\hat{y},u,v)$ is Lebesgue measurable for each
$(\hat{y},u,v)$ and $\hat{f}(t,.,.,.)$ is continuous a.e. $t\in[t_{0},t_{1}];$

\item[$H2)$] there exist integrable functions $\psi(.)$ and $\chi(.)$
such that
\[
|\hat{f}(t,\hat{y},u,v)-\hat{f}(t,\hat{y}',u,v)|\leq\psi(t)|\hat{y}-\hat{y}'|
\]
and 
\[
|\hat{f}(t,\hat{y},u,v)|\leq\chi(t)
\]
for every $\hat{y},\hat{y}'\in\hat{\Omega},$ $u\in U,$ $v\in V$,
a.e. $t\in[t_{0},t_{1}];$

\item[$H3)$] there exist positive numbers $L_{h_{0}},L_{h_{1}},L_{\hat{h}}\geq0$
such that
\[
|h_{i}(y)-h_{i}(y')|\leq L_{h_{i}}|y-y'|,\; i=0,1,\qquad|\hat{h}(\hat{y})-\hat{h}(\hat{y}')|\leq L_{\hat{h}}|\hat{y}-\hat{y}'|,
\]
for every $y,y'\in \Omega$ and $\hat{y},\hat{y}'\in\hat{\Omega}.$

\end{itemize}

The following remark simplifies problem $(OP)$ without loss of generality.
We introduce the new time independent variable $\tau\in[\tau_{0},\tau_{1}]$
and, given an integrable function $\phi(.)\geq1$ on $[\tau_{0},\tau_{1}],$
we define the function 
\[
t(\tau):=t_{0}+\intop_{\tau_{0}}^{\tau}\phi(s)ds,
\]
 for every $\tau\in[\tau_{0},\tau_{1}].$ Since $\phi(s)\geq1$ a.e.
$s\in[t_{0},t_{1}],$ then $t(.)$ is an increasing function and there
exists an inverse $\tau(t).$ If we suppose that the function $\hat{f}':[\tau_{0},\tau_{1}]\times \hat{\Omega}\times U\times V\rightarrow\mathbb{R}^{n+m}$
satisfies the hypotheses $H2)$ with Lipschitz constant $\phi(.),$
then we can set $\hat{g}(t,y,u,v):=[\phi(\tau(t))]^{-1}\hat{f}'(\tau(t),y,u,v)$
and $t_{1}:=t(\tau_{1}).$ It turns out that 
\[
\intop_{t_{0}}^{t_{1}}\hat{g}(s,y,u,v)ds=\intop_{\tau_{0}}^{\tau_{1}}\hat{f}'(s,y,u,v)ds
\]
and that $[\phi(\tau(t))]^{-1}\leq1.$ This implies that the new function
$\hat{g}(t,.,u,v)$ is Lipschitz continuous with a constant $L_{\hat{g}}\leq1.$

From now on, we suppose that this time transformation has been
already carried out on the function $\hat{f}$ and that there exists
a Lipschitz constant $L_{\hat{f}}\leq1.$

\section{Relaxed and Hyperrelaxed Problems}
\label{sec:Relaxed-and-Hyper-relaxed}

The adverse control problem $(OP)$ does not always admit a solution,
when we restrict the choice of controls $u(.)$ and $v(.)$ to be elements of $\mathcal{U}$ and $\mathcal{V},$ respectively.
As it is showed in \cite{key-4} and \cite{key-5}, there are two
ways to guarantee the existence of a solution for adverse control problems.

The first method concerns the symmetric relaxation of both players.
More precisely, we introduce the set of Borel measurable mappings
\[
\mathcal{S}:=\left\{ \sigma(.):[t_{0},t_{1}]\rightarrow\mathrm{r.p.m.}(U)\,:\;\sigma(t)(U(t))=1\; a.e.\; t\in[t_{0},t_{1}]\right\} ,
\]
where $\mathrm{r.p.m.}(U)$ is the set of the Radon probability measure
on $U$, and we symmetrically extend the choice for the second player
to the set of Borel measurable mappings
\[
\mathcal{S}_{P}:=\left\{ \sigma_{P}(.):[t_{0},t_{1}]\rightarrow\mathrm{r.p.m.}(V):\;\sigma_{P}(t)(V(t))=1\; a.e.\; t\in[t_{0},t_{1}]\right\} .
\]
Then we consider the new relaxed problem $(RP),$ which has the same
data of $(OP),$ but where the dynamic equations are replaced by
\[
y(\sigma)(t)=b+\intop_{t_{0}}^{t}ds\intop f(s,y(s),u)\sigma(s)(du)
\]
\[
\tilde{y}(\sigma\times\sigma_{P})(t)=b_{P}+\intop_{t_{0}}^{t}ds\intop\tilde{f}(s,\hat{y}(s),u,v)\sigma(s)(du)\times\sigma_{P}(s)(dv),
\]
where $\sigma\times\sigma_{P}$ is the product measure between $\sigma\in\mathcal{S}$
and $\sigma_{P}\in\mathcal{S}_{P}.$ The control strategies for players one and two are now elements of $\mathcal{S}$ and $\mathcal{S}_{P}$, respectively. It is showed in (\cite{key-4},
Example IX.2.2, pp 453-456) that the problems $(OP)$ and $(RP)$ can
have different values if we do not assume some special hypotheses on
the structure of the dynamic. 

We now move our attention to the hyperrelaxed extension, which does
not modify the value function of the problem, let alone special
assumptions on the structure of $\hat{f}.$ The problem is modified
as follows: the first player can still choose the control strategy in the set of
relaxed controls $\mathcal{S}$ while the second player, in order to
not modify the cost of the problem, has to pick controls up from a larger
set then $\mathcal{S}.$ These controls are mentioned as hyperrelaxed
and lie in the set
\[
\tilde{\mathcal{P}}:=\left\{ \pi(.,.):[t_{0},t_{1}]\times U\rightarrow\mathrm{r.p.m.}(V):\;\pi(t,u)(V(t))=1,\; a.e.\: t\in[t_{0},t_{1}],\;\forall\, u\in U\right\} .
\]
Roughly speaking, if we consider the set of Borel measurable mappings
\[
\mathcal{Q}:=\{\alpha:[t_{0},t_{1}]\rightarrow\mathrm{r.p.m.}(U\times V)\,:\alpha(U(t)\times V(t))=1,\; a.e.\: t\in[t_{0},t_{1}]\},
\]
then $\mathcal{S}$ can be considered as the set of the Borel measurable
mappings from $[t_{0},t_{1}]$ to the set of marginal probabilities on $U,$
while $\tilde{\mathcal{P}}$ can be regarded as the set of the mappings
from $[t_{0},t_{1}]$ to the set of the conditional probabilities on $V$
with respect to the information $u\in U.$

The hyperrelaxed problem has the same data of $(OP)$, but the
choice of controls is $\sigma\in\mathcal{S}$ for the first player
and $\pi$ belonging to a modification of $\tilde{\mathcal{P}}$ (details will be given later in the paper) for the second one. This changes
the dynamic equations as follows:
\[
y(\sigma)(t)=b+\intop_{t_{0}}^{t}ds\intop f(s,y(\sigma)(s),u)\sigma(s)(du),
\]
\[
\tilde{y}(\sigma\otimes\pi)(t)=b_{P}+\intop_{t_{0}}^{t}ds\intop\tilde{f}(s,\hat{y}(\sigma\otimes\pi)(s),u,v)\sigma(s)(du)\otimes\pi(u,s)(dv),
\]
where the symbol $\sigma\otimes\pi$ denotes the unique element in
$\mathcal{Q}$ such that 
\[
\intop_{t_{0}}^{t}ds\intop\varphi(s,u,v)\sigma(s)(du)\otimes\pi(u,s)(dv)=\intop_{t_{0}}^{t}ds\intop\sigma(s)(du)\intop\varphi(s,u,v)\pi(u,s)(dv)
\]
for every $\varphi(.,.,.)\in L^{1}(dt,C(U\times V)),$ (for more
details, see \cite{key-4}, Lemma X.1.3, pp. 485). We denote as
$(HP)$ the hyperrelaxed version of the problem stated in section
$2.$

As it is pointed out in (\cite{key-4}, Remark, pp. 489), there appears
not useful way to define a compact metric topology on $\tilde{\mathcal{P}}$
such that the function $\pi\mapsto\sigma\otimes\pi$ is continuous
for every $\sigma\in\mathcal{S}.$ However, we can overcome this difficult
proceeding as follows:
\begin{enumerate}
\item Restrict our attention to any denumerable subset $\mathcal{S}'\subset\mathcal{S};$
\item Introduce on $\tilde{\mathcal{P}}$ the following relation $\sim:$
$\pi_{1}\sim\pi_{2}$ if and only if
\[
\sigma\otimes\pi_{1}(t)=\sigma\otimes\pi_{2}(t)\qquad a.e.\: t\in[t_{0},t_{1}],\qquad\forall\,\sigma\in\mathrm{co}\,\mathcal{S}'.
\]
We denote as $\mathcal{P}$ the set of equivalence classes
on $\tilde{\mathcal{P}};$
\item We introduce a compact metric topology on $\mathcal{P},$ which makes
continuous the mapping $\pi\mapsto\sigma\otimes\pi$ for every $\sigma\in\mathrm{co}\,\mathcal{S}'.$
By (\cite{key-4}, Lemma X.I.I, pp. 482), for every $\tilde{\sigma}\in\mathrm{co}\,\mathcal{S}'$
there exists a unique nonatomic measure $\tilde{\zeta}$ such that
\[
\intop_{t_{0}}^{t_{1}}dt\intop h(t,u)\tilde{\sigma}(t)(du)=\intop h(t,u)\tilde{\zeta}(d(t,u))\qquad\forall\: h\in L^{1}([t_{0},t_{1}]\times U, \tilde{\zeta}).
\]
$\mathcal{P}$ can be seen as the set of the $\tilde{\zeta}-$measurable
mappings $\pi:[t_{0},t_{1}]\times U\rightarrow\mathrm{r.p.m}.(V)$ such that
$\pi(t,u)(V(t))=1$ $\tilde{\zeta}-$a.e. which elements are also Borel measurable on the set
$([t_{0},t_{1}]\times U)$. (for more details, see (\cite{key-4},
Definition X.2.1, pp. 496) and following discussion).
\end{enumerate}
We now state a lemma that brings the link between the
measure $\tilde{\zeta}$ and the elements $\sigma\in\mathcal{S}'$ to light.
\begin{lem}
\label{lem:X.2.2 Warga}Let $\mathcal{S}',$ $\mathcal{P}$ and $\tilde{\zeta}$
be defined as above. Then $\lim_{i}\sigma\otimes\pi_{i}=\sigma\otimes\pi$
for every $\sigma\in\mathcal{S}'$ if $\lim_{i}\pi_{i}=\pi$ in $\mathcal{P}.$
Furthermore, if $E$ is a $\tilde{\zeta}-$null set, then $\chi_{E}(t,u)=0$
$\sigma(t)-$a.e. r,  a.e. $t\in[t_{0},t_{1}]$, for every $\sigma\in\mathcal{S}'.$
\end{lem}
\begin{proof}
See (\cite{key-4}, Lemma X.2.2, pp. 497).
\end{proof}
Using lemma \ref{lem:X.2.2 Warga} and the formal construction of the
hyperrelaxed control set described in points $(1)-(3)$, it is easy
to check that the definition of hyperrelaxed controls does not depend
on the choice of $\tilde{\zeta}.$ Indeed, taking any other $\sigma\in\mathrm{co}\mathcal{S}'$
and the associated measure $\zeta,$ it turns out that $\zeta$ and
$\tilde{\zeta}$ are equivalent and  null sets of $\tilde{\zeta}$
are also null sets of $\zeta$ (and vice versa).

It is proven in (\cite{key-4}, Theorem VI.I.I, pp. 348) that the
functions $\mathcal{S}\ni\sigma\mapsto y(\sigma)$ and $\mathcal{Q}\ni\sigma\otimes\pi\rightarrow\hat{y}(\sigma\otimes\pi)$
are continuous. Furthermore, since (\cite{key-4}, Lemma X.3.3, pp. 504),
the function $\mathcal{P}\ni\pi\mapsto\hat{y}(\sigma\otimes\pi)$
is continuous for every $\sigma\in\mathrm{co}\mathcal{S}',$ and also
the function $\mathcal{S}_{P}\ni\sigma_{P}\mapsto\sigma\otimes\sigma_{P}$
is continuous for every $\sigma\in\mathcal{S}.$\\
To summarize, the new relaxed and hyper-relaxed adverse control problems can be written as follows:

\[
(RP)\,\left\{ \begin{array}{l}
\mathrm{Minimize}_{\sigma\in \mathcal{S}} \,h_{0}(y(\sigma)(t_{1}))\\
\mathrm{over}\quad \sigma \times \sigma_{P}\in \mathcal{S}\times \mathcal{S}_{P},\; \mathrm{s.t.}\\
\dot{y}(t)=f(t,y(t),\sigma(t))\quad\mathrm{a.e.}\: t\in[t_{0},t_{1}]\\
\dot{\tilde{y}}(t)=\tilde{f}(t,\hat{y}(t),\sigma\times\sigma_{P}(t)) \quad \mathrm{a.e.}\: t\in[t_{0},t_{1}]\\
y(t_{0})=b\in B,\quad \tilde{y}(t_{0})=\tilde{b}\in \tilde{B}\quad\mathrm{and}\quad h_{1}(y(\sigma)(t_{1}))=0\\
\hat{h}(\hat{y}(\sigma \times \sigma_{P})(t_{1}))\leq 0 \quad \mathrm{ for\: each}\:\ \sigma_{P}\in \mathcal{S}_{P}
\end{array}\right. ,
\]

and

\[
(HP)\,\left\{ \begin{array}{l}
\mathrm{Minimize}_{\sigma\in \mathcal{S}} \,h_{0}(y(\sigma)(t_{1}))\\
\mathrm{over}\quad \sigma \in \mathcal{S}, \quad \pi \in \mathcal{P},\; \mathrm{s.t.}\\
\dot{y}(t)=f(t,y(t),\sigma(t))\quad\mathrm{a.e.}\: t\in[t_{0},t_{1}]\\
\dot{\tilde{y}}(t)=\tilde{f}(t,\hat{y}(t),\sigma\otimes\pi(t)) \quad \mathrm{a.e.}\: t\in[t_{0},t_{1}]\\
y(t_{0})=b\in B,\quad \tilde{y}(t_{0})=\tilde{b}\in \tilde{B}\quad\mathrm{and}\quad h_{1}(y(\sigma)(t_{1}))=0\\
\hat{h}(\hat{y}(\sigma \otimes \pi)(t_{1}))\leq 0 \quad \mathrm{ for\: each}\:\pi\in \mathcal{P}
\end{array}\right. ,
\]
where, to the sake of shortness,  we have used the notation
$$f(t,y(t),\sigma(t))=\intop f(t,y(\sigma)(t),u)\sigma(t)(du),$$
$$\tilde{f}(t,\hat{y}(t),\sigma\otimes\pi(t))=\intop\tilde{f}(t,\hat{y}(\sigma\otimes\pi)(t),u,v)\sigma(t)(du)\otimes\pi(u,t)(dv)$$
and  
$$\tilde{f}(t,\hat{y}(t),\sigma\times\sigma_{P}(t))=\intop\tilde{f}(t,\hat{y}(\sigma\times\sigma_{P})(t),u,v)\sigma(du)\times\sigma_{P}(dv)(t)$$
\section{A special choice of $\mathcal{S}'$}
\label{sec:Choice Den Set}

Let $\bar{\sigma}\in\mathcal{S}$ be a given relaxed control. We define
a special denumerable set $\mathcal{S}'$ which will be used at a
succeeding stage. By (\cite{key-4} Condition IV.3.I, pp 280), there
exists a denumerable subset $\mathcal{U}_{\infty}$ of $\mathcal{U}$
such that the set $\left\{ u(t)|\, u\in\mathcal{R}_{\infty}\right\} $
is dense in $R(t)$ a.e. $t\in[t_{0},t_{1}].$ We stress that the
compactness of $V$ is a sufficient condition by which such condition
is satisfied. If we denote as $I_{\infty}$ the set of all the subintervals
$[a,b]$ of $[t_{0},t_{1}]$ with rational endpoints, then the set
$\mathcal{U}_{\infty}\times I_{\infty}$ is still denumerable and
takes the form $\{(u^{j},[a^{j},b^{j}]):\, j\in\mathbb{N}\}.$ We denote
as $\delta_{r}$ the Dirac measure at $r;$ we set $\sigma_{0}:=\bar{\sigma}$
and, for all $j\in\mathbb{N},$ 
\[
\sigma_{j}(t):=\left\{ \begin{array}{lll}
\delta_{u^{j}(t)} &  & if\; t\in[a^{j},b^{j}]\\
\bar{\sigma}(t) &  & otherwise
\end{array}\right..
\]
We finally set $\mathcal{S}':=\{\sigma_{0},\sigma_{1},\sigma_{2},\ldots\}.$
This special construction for the denumerable set $\mathcal{S}'$
will be helpful in the proofs of theorems \ref{thm:Hyper-Relaxed}, \ref{thm:Relaxed Prob}.

\section{Relaxed Derivatives}
\label{sec:Relaxed-Derivatives}

Consider an open set $\Omega'\subset\mathbb{R}^{n}$ and a set $\Omega\subset \Omega'$
which has compact closure into $\Omega.$ We use the notation $\Omega\subset\subset \Omega'$
and we can always suppose that there exists an $\varepsilon>0$ small enough such
that $\Omega+\varepsilon \mathbb{B}\subset \Omega'.$ 

We construct a $C^{\infty}$ function as follows: define 
\[
\bar{\varrho}(x):=\begin{cases}
\begin{array}{lll}
\exp\left(\frac{-1}{1-|x|^{2}}\right) &  & if\; x\in \mathbb{B}\\
\\
0 &  & otherwise
\end{array}\end{cases}
\]
and consider its normalization $\varrho(x)=\frac{\bar{\varrho}(x)}{\int_{\mathbb{B}}\bar{\varrho}(x')dx'}.$
It follows that $\varrho(.)$ is $C^{\infty},$ has compact support
in $\mathbb{B}$ and 
\[
\intop_{\mathbb{R}^{n}}\varrho(x)dx=\intop_{\mathbb{B}}\varrho(x)=1.
\]
Furthermore we can define a function $\varrho^{j}(x):=j^{n}\varrho(jx)$
which has compact support in $\frac{1}{j}\mathbb{B}$ and such that $\int\varrho^{j}(x)dx=1$
for each $j\in\mathbb{N}.$ We say that the function $\varrho(.)$
is a $C^{\infty}$ mollifier and that $\{\varrho^{j}\}$ is a sequence
of mollifiers.

Let $\phi:\Omega'\subset\mathbb{R}^{n}\rightarrow\mathbb{R}$ be a locally
Lipschitz function with constant $L_{\phi}$ and take a point $x\in \Omega.$
We next consider the convolution between the sequence of mollifiers
and the function $\phi(.)$ defining
\[
\phi^{j}(x):=(\phi*\varrho^{j})(x):=\intop_{j^{-1}\mathbb{B}}\varrho^{j}(y)\phi(x-y)dy=\intop_{x+j^{-1}\mathbb{B}}\varrho^{j}(x-y)\phi(y)dy.
\]
The last equality is well defined for $j$ sufficiently small because
$x\in \Omega.$ The sequence $\{\phi^{j}(.)\}_{j\in\mathbb{N}}$ is called
Fredholm approximation of the function $\phi(.).$ It turns out that
the functions $\phi^{j}(.)$ are $C^{\infty}$ and their partial
derivatives $\partial\phi^{j}(.)$ are uniformly continuous on $\bar{\Omega}$.
Furthermore, by the Rademacher theorem, the function $\partial_{x}\phi(.)$
exists a.e. and, for $x\in \Omega$, we set
\[
\partial_{x}\phi^{j}(x):=\intop_{j^{-1}\mathbb{B}}\varrho^{j}(y)\partial_{x}\phi(x-y)dy.
\]
If $\eta:[t_{0},t_{1}]\rightarrow \Omega$ is a continuous function, then
$t\mapsto\partial_{x}^{j}\phi(\eta(t))$ is Lebesgue measurable and dominated
by $|\partial_{x}^{j}\phi(\eta(t))|\leq L_{\phi}.$ It follows that the integral
\[
\Phi^{j}:=\intop_{t_{0}}^{t_{1}}\partial_{x}^{j}\phi(\eta(t))dt
\]
is well defined and, by the dominated convergence theorem, there exists $\Phi:=\lim_{j}\Phi^{j}.$
We call $\Phi$ the relaxed derivative of $\phi(.)$ evaluated along
the continuous function $\eta(.).$ 

We have the following result:
\begin{lem}
\label{lem:Convolution_1}Suppose that $\phi(.)$ and $\Omega$ are the
same objects defined in section \ref{sec:Relaxed-Derivatives} and that $\phi^{j}(.)$
are the Fredholm approximations of $\phi(.)$. Then $\phi^{j}(.)$
are locally Lipschitz with constant $L_{\phi}$ and, for all $x\in \Omega,$
\[
|\partial_{x}\phi^{j}(x)|\leq L_{\phi},\qquad|\phi^{j}(x)-\phi(x)|\leq L_{\phi}/j.
\]
\end{lem}

\begin{proof} See Brezis.
\end{proof}

The analysis carried out in this section can be easily extended to
any function $\phi:[t_{0},t_{1}]\times \Omega \times U\rightarrow\mathbb{R}^{n}$
such that $t\mapsto\phi(t,x,u)$ is integrable for every $(x,u)\in \Omega \times U$,
$x\mapsto\phi(t,x,u)$ is Lipschitz continuous a.e. $t\in[t_{0},t_{1}],$ for
every $u\in U$ and $u\mapsto\phi(t,x,u)$ is continuous a.e. $t\in[t_{0},t_{1}],$
for every $x\in \Omega.$

\section{Preliminary Results: The smooth case}

In this section, we state results dealing with the case in which all the data  of the adverse control problems $(RP)$ and $(HP)$ are smooth. Such lemmas are very similar to \cite{key-4}, X.3.5, X.3.7, but they differs in the typology of adverse control problems that we are dealing with. In what follows, we will invoke the following hypothesis:

\begin{itemize}
\item[$H4)$]   $\hat{f}(t,.,u,v)$, $\hat{h}(.)$ and $h_{i}(.)$, $i=0,1$ are all continuously differentiable, a.e. $t\in [t_{0},t_{1}]$, for all $u\in U$, $v\in V$. 
\end{itemize}
Furthermore, for $t\in [t_{0},t_{1}]$, $\bar{\sigma}\in \mathcal{S}$ and  $\pi \in \mathcal{P}$, we will denote as
\[
Z(t):=I_{n}+\intop_{t}^{t_{1}}Z(s)\partial_{x}f(s,y(\bar{\sigma})(s),\bar{\sigma}(s))ds,
\]
\[
\hat{Z}(\pi)(t)=I_{n+m}+\intop_{t}^{t_{1}}\hat{Z}(\pi)(s)\partial_{x}\hat{f}(s,\hat{y}(\bar{\sigma}\otimes\pi)(s),(\bar{\sigma}\otimes\pi)(s))ds,
\]
which are well posed since the regularity of the dynamics with respect to the state variable.
\begin{lem}\label{smooth_hyper}
Assume hypotheses $H1)$ - $H4)$. Then, given $(\bar{\sigma},\bar{b}, \bar{\tilde{b}})$ minimizer for problem $(HP)$, then there exist $l_{0}\geq 0$, $l_{1}\in \mathbb{R}^{n}$ and $\omega \in \mathrm{frm}^{+}(\mathcal{P})$ such that
\[ 
i) \qquad 0< l_{0}+l_{1} + \omega(\mathcal{P}) \leq 1;
\]
if we set
\[
k(t):=(l_{0}\partial_{x}h_{0}(y(\bar{\sigma})(t_{1})) + l_{1}\partial_{x}h_{1}(y(\bar{\sigma})(t_{1})))Z(t),
\]
\[ \hat{k}(\pi)(t):=\partial_{x} \hat{h}(\hat{y}(\bar{\sigma}\otimes \pi)(t_{1}))\hat{Z}(\pi)(t),
\]
$$ \mathfrak{h}(\pi, t, u)= \max_{v\in V(t)} \hat{k}(\pi)(t)\hat{f}(t,\hat{y}(\bar{\sigma}\otimes\pi)(t), u, v),$$
and
$$H(t,u)=k(t)f(t,y(\bar{\sigma})(t),u) + \int \mathfrak{h}(\pi, t, u) \omega(d\pi),$$
for each $\pi \in \mathcal{P}$, $u\in U(t)$, a.e. $t\in [t_{0}, t_{1}]$, then
\[
ii )\qquad\intop H(t,u)\bar{\sigma}(t)(du)=\min_{u\in U(t)}H(t,u)\quad a.e.\; t\in[t_{0},t_{1}],
\]
\[
iii)\qquad \mathfrak{h}(\pi,t,u)=\intop \hat{k}(\pi)(t)\hat{f}(t,\hat{y}(\bar{\sigma}\otimes\pi)(t),u,v)\pi(t,u)(dv),
\]
for $\omega$-a.a. $\:\pi\in\mathcal{P}$,
\[
iv)\qquad\hat{h}(\hat{y}(\bar{\sigma}\otimes\pi^{*})(t_{1})=\max_{\pi\in\mathcal{P}}\hat{h}(\hat{y}(\bar{\sigma}\otimes\pi)(t_{1}))=0
\]
for $\omega$-a.a. $\:\pi^{*}\in\mathcal{P}$, and
\[
v) \quad k(t_{0})Z(t_{0})\bar{b}+ \lambda(\bar{b},\bar{\tilde{b}})= \min_{(b,\tilde{b})\in \hat{B}}k(t_{0})Z(t_{0})b+ \lambda (b,\tilde{b}),
\]
where $\lambda:=\int \hat{k}(\pi)(t_{0})\omega(d\pi)$.
\end{lem}

\begin{proof}
We specialize theorem \cite{key-4}, X.2.4 to the data of the problem $(HP)$.
From condition \cite{key-4}, X.2.4 (1), it follows that there exist $l_{0}\geq 0$, $l_{1}\in \mathbb{R}^{n}$, $\omega \in \mathrm{frm}^{+}(\mathcal{P})$ and an $\tilde{\omega}(.)$ function which is $L^{1}(\omega,\mathcal{P})$ and such that $|\tilde{\omega}(\pi)|=1$ for $\omega$-a.a. $\:\pi\in\mathcal{P}$. In particular, from condition \cite{key-4}, X.2.4 (3), we obtain that $\tilde{\omega}\equiv 1$ and that
$$\hat{h}(\hat{y}(\bar{\sigma}\otimes\pi^{*})(t_{1})=0.$$
Conditions $i)$ and $iv)$ are then satisfied.\\
Now applying the result of \cite{key-4}, X.3.2 to the function $(\sigma \otimes \pi) \mapsto \hat{h}(\hat{y}(\sigma \otimes \pi)(t_{1}))$ and combining with \cite{key-4}, X.1.4, we get the relation
$$\mathfrak{h}(\pi,t,u)=\intop\partial_{x} \hat{h}(\hat{y}(\bar{\sigma}\otimes \pi)(t_{1}))\hat{Z}(\pi)(t)\cdot\hat{f}(t,\hat{y}(\bar{\sigma}\otimes\pi)(t),u,v)\pi(t,u)(dv),$$
for $\omega$-a.a. $\:\pi\in\mathcal{P}$, which is condition $iii)$.\\
From \cite{key-4}, X.2.4 (2), arguing as in the proof of \cite{key-4}, X.3.5, relations $ii)$ and $v)$ follow. This completes the proof.
\end{proof}

We now state a similar result for relaxed adverse control problems. We do not perform the proof since it is based on the same arguments of the proof of Lemma \ref{smooth_hyper}. In this case, for given $t\in [t_{0},t_{1}]$, $\bar{\sigma}\in \mathcal{S}$ and  $\bar{\sigma}_{P}\in \mathcal{S}$, the function $Z(.)$ remains unchanged, while we define
\[
\hat{Z}(\sigma_{P})(t)=I_{n+m}+\intop_{t}^{t_{1}}\hat{Z}(\sigma_{P})(s)\partial_{x}\hat{f}(s,\hat{y}(\bar{\sigma}\times\sigma_{P})(s),(\bar{\sigma}\times\sigma_{P})(s))ds.
\]

\begin{lem}\label{smooth_relaxed}
Assume hypotheses $H1)$ - $H4)$. Then, given $(\bar{\sigma},\bar{b})$ minimizer for problem $(RP)$, then there exist $l_{0}\geq 0$, $l_{1}\in \mathbb{R}^{n}$ and $\omega \in \mathrm{frm}^{+}(\mathcal{P})$ such that
\[ 
i) \qquad 0< l_{0}+l_{1} + \omega(\mathcal{P}) \leq 1;
\]
if we set 
\[
k(t):=(l_{0}\partial_{x}h_{0}(y(\bar{\sigma})(t_{1})) + l_{1}\partial_{x}h_{1}(y(\bar{\sigma})(t_{1})))Z(t),
\]
\[ 
\hat{k}(\sigma_{P})(t):=\partial_{x}\hat{h}(\hat{y}(\bar{\sigma}\times \sigma_{P})(t_{1}))\hat{Z}(\sigma_{P})(t)
\]
\[
\mathfrak{h}(\sigma_{P},t,s):=\max_{v\in V(t)}\hat{k}(\sigma_{P})(t)\cdot\intop\hat{f}(t,\hat{y}(\bar{\sigma}\times\sigma_{P})(t),u,v)s(du),
\]
and
$$H(t,s)=k(t)f(t,y(\bar{\sigma})(t),s) + \int \mathfrak{h}(\sigma_{P}, t, s) \omega(d\sigma_{P}),$$
then
\[
ii )\qquad\intop H(t,u)\bar{\sigma}(t)(du)=\min_{u\in U(t)}H(t,u)\quad a.e.\; t\in[t_{0},t_{1}],
\]
\[
iii)\qquad \mathfrak{h}(\sigma_{P},t,\bar{\sigma}(t))=\hat{k}(\sigma_{P})(t)\intop\hat{f}(t,\hat{y}(\bar{\sigma}\times\sigma_{P})(t),\bar{\sigma}(t),v)\sigma_{P}(dv),
\]
for $\omega$-a.a. $\:\sigma_{P}\in\mathcal{S}_{P}$, a.e. $t\in [t_{0},t_{1}]$,
\[
iv)\qquad\hat{h}(\hat{y}(\bar{\sigma}\times\sigma_{P}^{*})(t_{1}))=\max_{\sigma\in\mathcal{S_{P}}}\hat{h}(\hat{y}(\bar{\sigma}\times\sigma_{P})(t_{1}))=0
\]
for $\omega$-a.a. $\:\sigma_{P}^{*}\in\mathcal{S}_{P}$, and
\[
v) \quad k(t_{0})\bar{b}+\lambda (\bar{b}, \bar{\tilde{b}})= \min_{(b,\tilde{b}\in \hat{B}}k(t_{0})b+\lambda(b,\tilde{b}) 
\]
where $\lambda:=\int \hat{k}(\sigma_{P})(t_{0})\omega(d\sigma_{P})$.

\end{lem}

\section{Perturbed Problems}
\label{sec:Perturbed-Problem}

Fredholm approximations can be used to define a sequence of problems
whose limit approximates the behavior of $(OP)$. If we consider
the function $\hat{f}(t,y,u,v),$ we can construct its Fredholm
approximation with respect to $y$ as 
\[
\hat{f}^{j}(t,\hat{y},u,v)=\intop_{j^{-1}\mathbb{B}}\varrho^{j}(x)\hat{f}(t,\hat{y}-x,u,v)dx.
\]
The same procedure can be carried out on the functions
\[
h_{i}^{j}(y)=\intop_{j^{-1}\mathbb{B}}\varrho^{j}(x)h_{i}(y-x)dx,\qquad i=0,1,
\]
and 
\[
\hat{h}^{j}(\hat{y})=\intop_{j^{-1}\mathbb{B}}\varrho^{j}(x)\hat{h}(\hat{y}-x)dx.
\]

The next properties is helpful for the pursuance of the discussion:
\begin{lem}
\label{lem:Properties_Perturbed_Prob}Let
\[
\alpha:=\int_{t_{0}}^{t_{1}}\chi(\tau)d\tau,\quad c_{\hat{y}}:=L_{\hat{f}}+\alpha e^{\alpha},\quad c_{h_{i}}:=L_{h_{i}}(c_{y}+1)\;\forall\, i=0,1,\quad c_{\hat{h}}:=L_{\hat{h}}(c_{y}+1)
\]
Then, for each $\sigma\in\mathcal{S}$ and $\pi\in\mathcal{P}$ (or
$\pi\in\mathcal{S}_{P}$), we have:
\begin{itemize}

\item[$(i)$] $\;$ $w^{j}(t):=|\hat{y}^{j}(\sigma\otimes\pi)(t)-\hat{y}(\sigma\otimes\pi)(t)|\leq c_{\hat{y}}/j;$

\item[$(ii)$] $\;$ $|\hat{f}^{j}(t,\hat{y}^{j}(\sigma\otimes\pi)(t),(\sigma\otimes\pi)(t))-\hat{f}(t,\hat{y}(\sigma\otimes\pi)(t),(\sigma\otimes\pi)(t))|\leq(c_{\hat{y}}+1)/j;$

\item[$(iii)$] $\;$ $|h_{i}^{j}(y^{j}(\sigma)(t))-h_{i}(y(\sigma)(t))|\leq c_{h_{i}}/j\qquad\forall\, i=0,1;$

\item[$(iv)$] $\;$ $|\hat{h}^{j}(\hat{y}^{j}(\sigma\otimes\pi)(t))-\hat{h}(\hat{y}(\sigma\otimes\pi)(t))|\leq c_{\hat{h}}/j.$

\end{itemize}
for every $t\in[t_{0},t_{1}]$.
\end{lem}
\begin{proof}
To prove relation $(i),$ we fix $\sigma\in\mathcal{S}$ and $\pi\in\mathcal{P}$
(or $\pi\in\mathcal{S}_{P}$) and from the definition of $\hat{y}^{j}(\sigma\otimes\pi)(.),$
$\hat{y}(\sigma\otimes\pi)(.)$ respectively, it follows
\[
w^{j}(t):=|\hat{y}^{j}(\sigma\otimes\pi)(t)-\hat{y}(\sigma\otimes\pi)(t)|\leq
\]
\[
\leq\intop[|\hat{f}^{j}(s,\hat{y}^{j}(\sigma\otimes\pi)(s),\left(\sigma\otimes\pi\right)(s))-\hat{f}(s,\hat{y}^{j}(\sigma\otimes\pi)(s),\left(\sigma\otimes\pi\right)(s))|+
\]
\begin{equation}
+|\hat{f}(s,\hat{y}^{j}(\sigma\otimes\pi)(s),\left(\sigma\otimes\pi\right)(s))-\hat{f}(s,\hat{y}(\sigma\otimes\pi)(s),\left(\sigma\otimes\pi\right)(s))|]ds.\label{eq:Integral_Conv_2}
\end{equation}
By lemma \ref{lem:Convolution_1}, the first term of the integrand
is bounded by $j^{-1}L_{\hat{f}},$ while, in view of the Lipschitz
continuity of $\hat{f}(t,.,u,v),$ the second term of the integrand
is bounded by $L_{\hat{f}}w^{j}(t).$ From the Gronwall inequality,
relation $(i)$ follows. 

The proof of relations $(ii)-(iv)$ is consequence of relation $(i)$
and of the Lipschitz continuity of the functions $\hat{f}(t,.,u,v),$
$h_{0}(.),$ $h_{1}(.)$ and $\hat{h}(.).$ 
\end{proof}
It follows from lemma \ref{lem:Properties_Perturbed_Prob} that the
sequences $\{\hat{y}^{j}(\sigma\otimes\pi)(.)\}_{j\in J}$ and \\$\left\{ \hat{f}^{j}(t,\hat{y}^{j}(\sigma\otimes\pi)(.),(\sigma\otimes\pi)(.))\right\} _{j\in J}$
converge uniformly with respect to $\sigma\otimes\pi\in\mathcal{Q},$
a.e. $t\in[t_{0},t_{1}].$

In the following, we define problems that will be helpful in the proofs
of theorems \ref{thm:Hyper-Relaxed} and \ref{thm:Relaxed Prob}.

Suppose that $(RP)$ has a solution $(\bar{\sigma},\bar{b},\bar{b}_{P})\in\mathcal{S}\times B\times B_{P}.$
Then we consider the problem of seeking the optimal strategy $\sigma^{j}\in\mathcal{S}$
which minimizes the cost $h_{0}^{j}(y(\sigma)(t_{1}))$ and such that
\[
H_{1}^{j}(y(\sigma)(t_{1}))=0,\qquad\hat{H}^{j}(\hat{y}(\sigma\times\sigma_{P})(t_{1}))\leq0,\quad\forall\,\sigma_{P}\in\mathcal{S}_{P},
\]
where $y(\sigma)(.)$ and $\hat{y}(\sigma\times\sigma_{P})(.)$ are the
solutions of the equations
\begin{equation}
y(\sigma)(t)=\bar{b}+\intop_{t_{0}}^{t}ds\intop f^{j}(s,y(\sigma)(s),u)\sigma(s)(du),\label{eq:Pert_Eq_1}
\end{equation}
\begin{equation}
\hat{y}(\sigma\times\sigma_{P})(t)=(\bar{b},\bar{b}_{P})+\intop_{t_{0}}^{t}ds\intop\hat{f}^{j}(s,\hat{y}(\sigma\times\sigma_{P})(s),u,v)\sigma(s)(du)\times\sigma_{P}(s)(dv).\label{eq:Perturb_Eq_2}
\end{equation}
We define the functions 
\[
H_{1}^{j}(y^{j}(\sigma)(t_{1})):=h_{1}^{j}(y^{j}(\sigma)(t_{1}))-a^{j},
\]
for some suitable choice of $a^{j}\in\frac{c_{h_{1}}}{j}\mathbb{B},$ and 
\[
\hat{H}^{j}(\hat{y}^{j}(\sigma\times\sigma_{P})(t_{1})):=\hat{h}^{j}(\hat{y}^{j}(\sigma\times\sigma_{P})(t_{1}))+\frac{c_{\hat{h}}}{j}.
\]

At the light of the previous discussion, we denote as $(RP^{j})$ the following problem
\[
(RP^{j})\,\left\{ \begin{array}{l}
\mathrm{Minimize}_{\sigma\in \mathcal{S}} \,h_{0}^{j}(y(\sigma)(t_{1}))\\
\mathrm{over}\quad \sigma \times \sigma_{P}\in \mathcal{S}\times \mathcal{S}_{P},\; \mathrm{s.t.}\\
\dot{y}(t)=f^{j}(t,y(t),\sigma(t))\quad\mathrm{a.e.}\: t\in[t_{0},t_{1}]\\
\dot{\tilde{y}}(t)=\tilde{f}^{j}(t,\hat{y}(t),\sigma\times\sigma_{P}(t)) \quad \mathrm{a.e.}\: t\in[t_{0},t_{1}]\\
y(t_{0})=\bar{b},\quad \tilde{y}(t_{0})=\bar{\tilde{b}}\quad\mathrm{and}\quad H_{1}^{j}(y(\sigma)(t_{1}))=0\\
\hat{H}^{j}(\hat{y}(\sigma \times \sigma_{P})(t_{1}))\leq 0 \quad \mathrm{ for\: each}\:\ \sigma_{P}\in \mathcal{S}_{P}
\end{array}\right. ,
\]

If $(\bar{\sigma},\bar{b},\bar{b}_{P})$ is a solution for the problem
$(RP)$ stated in section $2,$ it is easy to check that, using lemma
\ref{lem:Properties_Perturbed_Prob}, we can choose the parameter
$a^{j}$ in such manner that $\bar{\sigma}$ is also an admissible
strategy for the perturbed problem $(RP^{j}),$ for every $j$ sufficiently
large. From general compactness arguments (see \cite{key-4}, Theorem
IX.1.1, pp 445), it follows that there exists a minimizing control
$\sigma^{j}\in\mathcal{S}$ that solves the problem $(RP^{j}).$ 

The same procedure can be carried out when the second player chooses
control strategies in $\mathcal{P}.$ In this case, equation (\ref{eq:Pert_Eq_1})
is not modified, while equation (\ref{eq:Perturb_Eq_2}) becomes
\begin{equation}
\hat{y}(\sigma\otimes\pi)(t)=(\bar{b},\bar{b}_{P})+\intop_{t_{0}}^{t}ds\intop\sigma(s)(du)\intop\hat{f}^{j}(s,\hat{y}(\sigma\otimes\pi)(s),u,v)\pi(s,u)(dv).\label{eq:Perturb_Hyper-Rel_eq}
\end{equation}
The functions $h_{0}^{j}(y(\sigma)(t_{1}))$, $H_{1}^{j}(y(\sigma)(t_{1}))$
and $\hat{H}(\hat{y}(\sigma\otimes\pi)(t_{1}))$ remain unchanged (we have just replaced $\sigma\times\sigma_{P}$ with $\sigma\otimes\pi$).\\
We define the problem of finding a control $\sigma\in\mathcal{S}$
which minimizes $h^{j}_{0}(y(\sigma)(t_{1})),$ such that $y(\sigma)(.)$
and $\hat{y}(\sigma\otimes\pi)(.)$ are solutions of (\ref{eq:Pert_Eq_1}),
(\ref{eq:Perturb_Hyper-Rel_eq}) and 
\[
H_{1}^{j}(y(\sigma)(t_{1}))=0,\qquad\hat{H}^{j}(\hat{y}(\sigma\otimes\pi)(t_{1}))\leq0\quad\forall\,\pi\in\mathcal{P}.
\]
We denote as $(HP^{j})$ the hyperrelaxed perturbed problem 
\[
(HP^{j})\,\left\{ \begin{array}{l}
\mathrm{Minimize}_{\sigma\in \mathcal{S}} \,h_{0}^{j}(y(\sigma)(t_{1}))\\
\mathrm{over}\quad \sigma \in \mathcal{S}, \quad \pi \in \mathcal{P},\; \mathrm{s.t.}\\
\dot{y}(t)=f^{j}(t,y(t),\sigma(t))\quad\mathrm{a.e.}\: t\in[t_{0},t_{1}]\\
\dot{\tilde{y}}(t)=\tilde{f}^{j}(t,\hat{y}(t),\sigma\otimes\pi(t)) \quad \mathrm{a.e.}\: t\in[t_{0},t_{1}]\\
y(t_{0})=\bar{b},\quad \tilde{y}(t_{0})=\bar{\tilde{b}}\quad\mathrm{and}\quad H^{j}_{1}(y(\sigma)(t_{1}))=0\\
\hat{H}^{j}(\hat{y}(\sigma \otimes \pi)(t_{1}))\leq 0 \quad \mathrm{ for\: each}\:\pi\in \mathcal{P}
\end{array}\right. ,
\]
and as $y^{j}(\sigma)(.)$ and $\hat{y}^{j}(\sigma\otimes\pi)(.)$
the solutions of (\ref{eq:Pert_Eq_1}), (\ref{eq:Perturb_Hyper-Rel_eq}) respectively.

The following remark is helpful in the proof of theorem \ref{thm:Hyper-Relaxed}.
Suppose we deal with the set of hyperrelaxed controls $\mathcal{P}$
with the particular choice of $\mathcal{S}'$ in section \ref{sec:Choice Den Set}
and assume that $(\bar{\sigma},\bar{b},\bar{b}_{P})$ is a solution
for the problem $(HP).$ From the particular
choice of $\mathcal{S}',$ it follows that $\bar{\sigma}$ is also
an admissible strategy for $(HP^{j})$ when we restrict our attention
to controls in $\mathrm{co}\mathcal{S}'$ for the first player. Furthermore,
since $\mathrm{co}\mathcal{S}'$ has the same properties of $\mathcal{S}$
(which means $\mathrm{co}\mathcal{S}'$ is convex and sequentially
compact), it follows that $(HP^{j})$ has also a solution $\sigma^{j}\in\mathrm{co}\mathcal{S}'$
for every $j$ (again, see \cite{key-4}, Theorem IX.1.1, pp 445).

\section{Main Theorems}

In the statement and the proof of theorem \ref{thm:Hyper-Relaxed},
we use the following notation. We denote as $y^{j}(\sigma)(.)$
the unique solution of equation (\ref{eq:Pert_Eq_1}) for $\sigma\in\mathrm{co}\mathcal{S}'$
and as $y^{j}(\sigma\otimes\pi)(.)$ the unique solution of (\ref{eq:Perturb_Hyper-Rel_eq})
for $\sigma\in\mathrm{co}\mathcal{S}'$ and $\pi\in\mathcal{P}$ (we
suppose that $\hat{\bar{b}}:=(\bar{b},\bar{b}_{P})$ is fixed for the perturbed problem).
From the discussion in section \ref{sec:Perturbed-Problem}, it follows that the problem
$(HP^{j})$ has a solution $\sigma^{j}\in\mathrm{co}\mathcal{S}'$
for every $j.$ We define the adjoint backward equations
\[
Z^{j}(t):=I_{n}+\intop_{t}^{t_{1}}Z^{j}(s)\partial_{x}f^{j}(s,y^{j}(\sigma^{j})(s),\sigma^{j}(s))ds,
\]
\[
\hat{Z}^{j}(\pi)(t)=I_{n+m}+\intop_{t}^{t_{1}}\hat{Z}^{j}(\pi)(s)\partial_{x}\hat{f}^{j}(s,\hat{y}^{j}(\sigma^{j}\otimes\pi)(s),(\sigma^{j}\otimes\pi)(s))ds.
\]
(To the sake of shortness, we have used the notation 
\[
\varphi(a,\sigma(t)):=\intop\varphi(a,u)\sigma(t)(du)
\] 
and
\[
\phi(a,\sigma\otimes \pi(t)):=\intop \sigma(t) (du)\intop\phi(a,u,v)\pi(t,u)(dv)
\] 
$a\in A$, where $A$ is a given set  and $\varphi:A\times R \rightarrow \mathbb{R}$, $\phi:A\times R \times v\rightarrow \mathbb{R}$ are given functions).
Since the function $x\mapsto\hat{f}^{j}(t,x,u,v)$ is $C^{1}$
for every $j,$ the integrals above are well defined.

In the convergence analysis of theorem \ref{thm:Hyper-Relaxed},  we deal with the  derivatives of the functions  $h_{1}^{j}(.)$
and $\hat{h}^{j}(.)$ instead of considering the derivatives of the functions $H_{1}^{j}(.)$
and $\hat{H}^{j}(.).$ It is easy to check that this simplification
does not affect the statements $i)-v)$ of the following theorem.

\begin{thm}
\label{thm:Hyper-Relaxed}Let $(\bar{\sigma},\bar{b},\bar{b}_{P})$
be an optimal solution to the problem $(HP)$. Then there exist a
set of index $J\subset\mathbb{N}$ , limiting multipliers $l_{0}\geq0,$
$l_{1}\in\mathbb{R},$ limiting initial directions $\mathcal{H}_{0},\mathcal{H}_{1}\in\mathbb{R}^{n},$
a $\omega\in\mathrm{f.r.m.^{+}}(\mathcal{P})$, a $\omega-$integrable
function $\hat{\mathcal{H}}:\mathcal{P}\rightarrow\mathbb{R}^{n+m},$
and, for each $\pi\in\mathcal{P},$ continuous functions $Z:[t_{0},t_{1}]\rightarrow M_{n\times n},$
$\hat{Z}(\pi):[t_{0},t_{1}]\rightarrow M_{\left(n+m\right)\times(n+m)},$
such that:

\[
i)\quad Z(t)=\lim_{j\in J}Z^{j}(t),\qquad\hat{Z}(\pi)(t)=\lim_{j\in J}\;\hat{Z}^{j}(\pi)(t)\quad \mathrm{uniformly} \;  t\in[t_{0},t_{1}], \; \pi\in\mathcal{P};
\]

\[
(l_{0},l_{1})=\lim_{j\in J}(l_{0}^{j},l_{1}^{j}),\quad\mathcal{H}_{0}=\lim_{j\in J}\partial_{x}h_{0}^{j}(y^{j}(t_{1})),\quad\mathcal{H}_{1}=\lim_{j\in J}\partial_{x}h_{1}^{j}(y^{j}(t_{1})),
\]
\[
\hat{\mathcal{H}}(\pi)=\lim_{j\in J}\partial_{x}\hat{h}^{j}(\hat{y}(\sigma^{j}\otimes\pi)(t_{1})),\qquad\omega-a.a.\quad\pi\in\mathcal{P};
\]

\[
ii)\qquad\quad0<l_{0}+\left|l_{1}\right|+\omega(\mathcal{P})\leq1
\]

Define:
\[
k(t):=\left(l_{0}\mathcal{H}_{0}+l_{1}\mathcal{H}_{1}\right)Z(t),\quad\hat{k}(\pi)(t):=\hat{\mathcal{H}}(\pi)\hat{Z}(\pi)(t),
\]
\[
\mathfrak{h}(\pi,t,u):=\max_{v\in v(t)}\hat{k}(\pi)(t)\cdot\hat{f}(t,\hat{y}(\bar{\sigma}\otimes\pi)(t),u,v),
\]

\[
H(t,u):=k(t)f(t,y(\bar{\sigma})(t),u)+\intop\mathfrak{h}(\pi,t,u)\omega(d\pi),
\]
Then
\[
iii)\qquad\intop_{t_{0}}^{t_{1}}dt\intop H(t,u)(\sigma-\bar{\sigma})(t)(du)\geq0\qquad\forall\sigma\in\mathcal{S}',
\]

\[
iv)\qquad\hat{h}(\hat{y}(\bar{\sigma}\otimes\pi^{*})(t_{1})=\max_{\pi\in\mathcal{P}}\hat{h}(\hat{y}(\bar{\sigma}\otimes\pi)(t_{1}))\quad\omega-a.a.\:\pi^{*}\in\mathcal{P}.
\]
Furthermore, since the choice of $\mathcal{S}'$ as in section \ref{sec:Choice Den Set},
condition $iii)$ can be strengthened, obtaining 
\[
v)\qquad\intop H(t,u)\bar{\sigma}(t)(du)=\min_{u\in U(t)}H(t,u)\quad a.e.\; t\in[t_{0},t_{1}].
\]
\end{thm}
\begin{proof}
\textbf{Step 0} We first show that it suffices to prove the theorem  in the special case where $B=\{ \bar{b}\}$ and $\tilde{B}=\{\bar{\tilde{b}}\}$. Indeed, suppose that $B^{*}$ and $\tilde{B}^{*}$ are two arbitrary convex and compact neighbourhood of $\bar{b}$ and $\tilde{\bar{b}}$, respectively and that theorem \ref{thm:Hyper-Relaxed} has been already proved with fixed initial conditions. Denote as $\hat{B}^{*}:=B^{*}\times \tilde{B}^{*}$. We now consider a new problem related to $(HP)$ 
\[
(MHP)\,\left\{ \begin{array}{l}
\mathrm{Minimize}_{\sigma\in \mathcal{S}} \,h_{0}(y(\sigma)(t_{1}))\\
\mathrm{over}\quad \sigma \in \mathcal{S}, \quad \pi \in \mathcal{P} \; \mathrm{and}\; (b(t), \tilde{b}(t))\in \hat{B}^{*}\;\mathrm{a.e}\; t\in[t_{0}-1,t_{1}]\;\mathrm{s.t.}\\
\dot{y}(t)= b(t)- \bar{b} \quad\mathrm{a.e.}\: t\in[t_{0}-1,t_{0}]\\
\dot{y}(t)=f(t,y(t),\sigma(t))\quad\mathrm{a.e.}\: t\in[t_{0},t_{1}]\\
\dot{\tilde{y}}(t)= \tilde{b}(t) - \bar{\tilde{b}}\quad\mathrm{a.e.}\: t\in[t_{0}-1,t_{0}]\\
\dot{\tilde{y}}(t)=\tilde{f}(t,\hat{y}(t),\sigma\otimes\pi(t)) \quad \mathrm{a.e.}\: t\in[t_{0},t_{1}]\\
y(t_{0})=b\in B,\quad \tilde{y}(t_{0})=\tilde{b}\in \tilde{B}\quad\mathrm{and}\quad h_{1}(y(\sigma)(t_{1}))=0\\
\hat{h}(\hat{y}(\sigma \otimes \pi)(t_{1}))\leq 0 \quad \mathrm{ for\: each}\;\pi\in \mathcal{P}
\end{array}\right. ,
\]

\textbf{Step 1}: Consider the sequences of functions
\[
\left\{ Z^{j}(t)=I_{n}+\intop_{t}^{t_{1}}Z^{j}(s)\partial_{x}f^{j}(s,y^{j}(\sigma^{j})(s),\sigma^{j}(s))ds\right\} _{j\in\mathbb{N}}
\]
and 
\[
\left\{ \hat{Z}^{j}(t)(\pi)=I_{n+m}+\intop_{t}^{t_{1}}\hat{Z}^{j}(\pi)(s)\partial_{x}\hat{f}^{j}(s,\hat{y}^{j}(\sigma^{j}\otimes\pi)(s),(\sigma^{j} \otimes \pi)(s))ds\right\} _{j\in\mathbb{N}}
\]
for every $\pi\in\mathcal{P}.$ It is easy to check that both sequences
are equibounded and equicontinuous. The first property follows because$ $
\[
|\partial_{x}f^{j}(t,y^{j}(\sigma^{j})(t),\sigma^{j}(t))|\leq L_{f},\qquad|\partial_{x}\hat{f}^{j}(t,y^{j}(\sigma^{j}\otimes\pi)(t),(\sigma^{j} \otimes \pi)(t))|\leq L_{\hat{f}}
\]
for every $j\in\mathbb{N},$ a.e. $t\in[t_{0},t_{1}],$ and therefore
\[
|Z^{j}(t)|\leq L_{f}\intop_{t_{0}}^{t_{1}}|Z^{j}(s)|ds+1,\qquad|\hat{Z}^{j}(\pi)(t)|\leq L_{\hat{f}}\intop_{t_{0}}^{t_{1}}|\hat{Z}^{j}(\pi)(s)|ds+1.
\]
 In view of the previous inequalities, we can use the Gronwall lemma
which yields
\[
|Z^{j}(s)|\leq1+(t_{1}-t_{0})L_{f}e^{(t_{1}-t_{0})},\quad|\hat{Z}^{j}(\pi)(s)|\leq1+(t_{1}-t_{0})L_{\hat{f}}e^{(t_{1}-t_{0})}.
\]
Therefore both sequences are uniformly bounded. Furthermore, by an
easy calculation, it follows that, for each $\pi\in\mathcal{P},$
\[
|\left(Z^{j}\right)'(t)|\leq|Z^{j}(t)|,\qquad|\left(\hat{Z}^{j}\right)'(\pi)(t)|\leq|\hat{Z}^{j}(\pi)(t)|,
\]
 a.e. $t\in[t_{0},t_{1}].$ These arguments
show that we can apply the Ascoli-Arzel\`a theorem and that there exist
$J_{1}\subset\mathbb{N}$ and continuous functions $Z(.),\hat{Z}(\pi)(.)$
such that  
\[
\lim_{j\in J_{1}}Z^{j}(t)=Z(t),\qquad\lim_{j\in J_{1}}\hat{Z}^{j}(\pi)(t)=\hat{Z}(\pi)(t),
\]
uniformly with respect to $t\in[t_{0},t_{1}]$ and $\pi\in\mathcal{P}.$ 

The perturbed problems $(HP^{j})$ have $C^{1}$ data and solutions
$\sigma^{j}\in\mathrm{co}\mathcal{S}'$. Therefore theorem (\cite{key-4},
X.3.5, pp 505) can be applied for the minimizer $(\sigma^{j},\bar{b},\bar{b}_{P})$.
In particular from (\cite{key-4}, X.3.5, pp 505, point (1)), it follows
that there exist $l_{0}^{j}\geq0,$ $l_{1}^{j}\in\mathbb{R}^{n}$
and $\omega^{j}\in\mathrm{f.r.m.}^{+}(\mathcal{P})$ such that, 
\[
0<l_{0}^{j}+|l_{1}^{j}|+\omega^{j}(\mathcal{P})\leq1.
\]
By standard compactness arguments, we can find a subset
$J_{2}\subset J_{1},$ $l_{0}\geq0,$ $l_{1}\in\mathbb{R}^{n}$ and
$\omega\in\mathrm{f.r.m.}^{+}(\mathcal{P})$ such that
\[
\lim_{j\in J_{2}}l_{0}^{j}=l_{0},\qquad\lim_{j\in J_{2}}l_{1}^{j}=l_{1},\qquad\lim_{j\in J_{2}}\omega^{j}=\omega\quad\mathrm{weakly-*}
\]
and
\[
0<l_{0}+|l_{1}|+\omega(\mathcal{P})\leq1.
\]

By lemma \ref{lem:Properties_Perturbed_Prob} $(i)$ and using the result
of (\cite{key-4}, Theorem VI.I.6, pp. 348), there exists $J_{3}\subset J_{2}$
such that
\begin{equation}
\lim_{j\in J_{3}}y^{j}(\sigma^{j})(t)=\lim_{j\in J_{3}}y(\sigma^{j})(t)=y(\bar{\sigma})(t) \label{eq:convergence_y} 
\end{equation}
and, since the continuity of the functions $\pi \mapsto \hat{y}^{j}(\sigma\otimes \pi)(.) $ for every $\sigma \in \mathrm{co} \mathcal{S}',$ 
\begin{equation}
\lim_{j\in J_{3}}\hat{y}^{j}(\sigma^{j}\otimes\pi)(t)=\lim_{j\in J_{3}}\hat{y}(\sigma^{j}\otimes \pi)(t)=\hat{y}(\bar{\sigma}\otimes\pi)(t) \label{eq:convergence_yhat}
\end{equation}
uniformly with respect to $\pi \in \mathcal{P}$ and $t\in[t_{0},t_{1}]$. (The fact that $\sigma^{j}\rightharpoondown \bar{\sigma}$ follows from the optimality of $\sigma^{j}$. Indeed
$$ \lim_{j\in J_{3}} h^{j}_{0}(y^{j}(\sigma^{j})(t_{1})=h_{0}(y(\bar\sigma)(t_{1})) \leq \lim_{j\in J_{3}} h^{j}_{0}(y^{j}(\sigma)(t_{1})=h_{0}(y(\sigma)(t_{1}))$$
for all $\sigma \in \mathrm{co}\mathcal{S}'$).

From lemma \ref{lem:Convolution_1}, it follows that $|\partial_{x}h_{0}^{j}(y^{j}(\sigma^{j})(t_{1}))|\leq L_{h_{0}},$
$|\partial_{x}h_{1}^{j}(y^{j}(\sigma^{j})(t_{1}))|\leq L_{h_{1}}$
and $|\partial_{x}\hat{h}(\hat{y}^{j}(\sigma^{j}\otimes\pi)(t_{1}))|\leq L_{\hat{h}}$
for every $\pi\in\mathcal{P}.$ Then, using a similar convergence
analysis, we can suppose that there exist $J\subset J_{3},$ $\mathcal{H}_{0},\mathcal{H}_{1}\in\mathbb{R}^{n}$
such that 
\[
\lim_{j\in J}\partial_{x}h_{0}^{j}(y^{j}(\sigma^{j})(t_{1}))=\mathcal{H}_{0},\qquad\lim_{j\in J}\partial_{x}h_{1}^{j}(y^{j}(\sigma^{j})(t_{1}))=\mathcal{H}_{1},
\]
and, using (\cite{key-4}, Lemma X.I.6 pp. 489) for each fixed $\pi\in\mathcal{P},$ there exist a vector $\hat{\mathcal{H}}(\pi)$ and a sequence $\{\pi^{j}\}_{j\in J}\subset \mathcal{P},$
such that 
\[
\lim_{j\in J}\partial_{x}\hat{h}^{j}(\hat{y}^{j}(\sigma^{j}\otimes\pi^{j})(t_{1}))=\hat{\mathcal{H}}(\pi).
\]

This completes the proof of points $i)$ and $ii).$

\textbf{Step 2:} We observe that the functions $\hat{\mathcal{H}}(.):\mathcal{P}\rightarrow\mathbb{R}^{n}$
and $\hat{Z}(.)(t):\mathcal{P}\rightarrow\mathbb{R}^{n+m}$ are pointwise 
limits of  sequences of continuous functions $\pi\mapsto\{\partial_{x}\hat{h}^{j}(\hat{y}^{j}(\sigma^{j}\otimes\pi)(t_{1}))\}_{j\in\mathbb{N}}$
and $\pi\mapsto\{\hat{Z}^{j}(\pi)(t)\}$; the latter converges uniformly
with respect to $\pi \in \mathcal{P}$. Since the set $\mathcal{P}$
is equipped with the Borel $\mathcal{B}(\mathcal{P})-$field, the continuous
functions $\pi\mapsto\{\partial_{x}\hat{h}^{j}(\hat{y}^{j}(\sigma^{j}\otimes\pi)(t_{1}))\}_{j\in\mathbb{N}}$
and $\pi\mapsto\{\hat{Z}^{j}(\pi)(t)\}$ are $\omega$-measurable
and their limit functions $\mathcal{H}(.)$ and $\hat{Z}(.)(t)$ are
also $\mathcal{B}(\mathcal{P})-$measurable.

From (\cite{key-4}, Theorem X.3.5, point (3)), it follows that

\[
\hat{h}^{j}(\hat{y}^{j}(\sigma^{j}\otimes\pi^{*})(t_{1}))=\max_{\pi\in\mathcal{P}}\hat{h}^{j}(\hat{y}^{j}(\sigma^{j}\otimes\pi)(t_{1})),\quad\omega^{j}-\mathrm{a.a.}\,\pi^{*}
\]
which is equivalent to
\[
\hat{h}^{j}(\hat{y}^{j}(\sigma^{j}\otimes\pi)(t_{1}))\leq\hat{h}^{j}(\hat{y}^{j}(\sigma^{j}\otimes\pi^{*})(t_{1}))\quad\omega^{j}-\mathrm{a.a.}\,\pi^{*},\:\forall\pi\in\mathcal{P}.
\]

Since $\omega^{j}$ is a positive measure for each $j,$ it follows
that 
\begin{equation}
\intop\left[\hat{h}^{j}(\hat{y}^{j}(\sigma^{j}\otimes\pi)-\hat{h}^{j}(\hat{y}^{j}(\sigma^{j}\otimes\pi^{*}))\right]\omega^{j}(d\pi^{*})\leq0,\qquad\forall\pi\in\mathcal{P}.\label{eq:pi_max_integral}
\end{equation}
We recall that, for $\sigma\in\mathrm{co}\mathcal{S}'$ and $\pi\in\mathcal{P},$
the function $\sigma\otimes\pi\mapsto\hat{h}^{j}(\hat{y}^{j}(\sigma\otimes\pi)(t_{1}))$
is continuous for every $j$, $\pi\mapsto\hat{h}^{j}(\hat{y}^{j}(\sigma^{j}\otimes\pi)(t_{1}))$
is continuous for every $j$ and $\omega^{j}\rightharpoondown\omega,$
$\sigma^{j}\rightharpoondown\bar{\sigma}$ weakly-{*}.
Adding and subtracting the term $\int \hat{h}^{j}(\hat{y}^{j}(\sigma^{j} \otimes \pi^{*})(t_{1}))\omega(d\pi^{*})$ (which is well-defined since the continuity properties of $\hat{h}(\hat{y}(\sigma\otimes\,.)(t_{1}))$), we can estimate
$$| \int \hat{h}^{j}(\hat{y}^{j}(\sigma^{j} \otimes \pi^{*})(t_{1}))(\omega-\omega^{j})(d\pi^{*}) | \leq |\sup_{\sigma\otimes \pi}\hat{h}(\hat{y}(\sigma \otimes \pi)(t_{1})) |\times |\int(\omega-\omega^{j})(d\pi^{*})| \rightarrow 0,$$
where the right hand side lets to 0 since $\omega^{j}\rightharpoondown \omega$ weakly-{*}. Then, with the help of the dominated convergence theorem, we can pass to the limit in (\ref{eq:pi_max_integral}), obtaining
\[
\intop\left[\hat{h}(\hat{y}(\bar{\sigma}\otimes\pi))-\hat{h}(\hat{y}(\bar{\sigma}\otimes\pi^{*}))\right]\omega(d\pi^{*})\leq0,\qquad\forall\pi\in\mathcal{P}.
\]

Define the sequences of functions $\hat{k}^{j}:\mathcal{P}\times[t_{0},t_{1}] \rightarrow \mathbb{R}^{n+m}$ and $\mathfrak{h}^{j}:\mathcal{P}\times[t_{0},t_{1}]\times R \rightarrow \mathbb{R}$ such that
\[
\hat{k}^{j}(\pi)(t):=\partial_{x}\hat{h}^{j}(\hat{y}^{j}(\sigma^{j}\otimes\pi)(t_{1}))\hat{Z}^{j}(\pi)(t)
\]
and 
\[
\mathfrak{h}^{j}(\pi,t,u):=\max_{v\in v(t)}\hat{k}^{j}(\pi)(t)\cdot\hat{f}^{j}(t,\hat{y}^{j}(\sigma^{j}\otimes\pi)(t),u,v).
\]
With the help of lemma \ref{lem:Properties_Perturbed_Prob}, it is
easy to check that, for every $\pi\in\mathcal{P},$ $\hat{k}^{j}(.)(\pi)\rightarrow\hat{k}(.)(\pi)$
in $L^{1}([t_{0},t_{1}],dt).$ Using again lemma \ref{lem:Properties_Perturbed_Prob},
it follows that $\mathfrak{h}^{j}(\pi,.,u)\rightarrow\mathfrak{h}(\pi,.,u)$
in $L^{1}$ for every $\pi\in\mathcal{P}$,  $r\in R,$ and $\mathfrak{h}^{j}(\pi,t,.)\rightarrow\mathfrak{h}(\pi,t,.)$
uniformly a.e. $t\in[t_{0},t_{1}],$ for every $\pi\in \mathcal{P}$ and $\mathfrak{h}^{j}(.,t,u)\rightarrow\mathfrak{h}(.,t,u)$ punctually, for every $r\in R$, a.e. $t\in [t_{0},t_{1}]$ 

By (\cite{key-4}, Theorem X.3.5, point (4)), the function $\pi\rightarrow\mathfrak{h}^{j}(\pi,t,u)$
is $\omega^{j}-$integrable for each $j,$ $r\in R,$ a.e. $t\in[t_{0},t_{1}]$
and can be expressed by the relation
\[
\mathfrak{h}^{j}(\pi,t,u)=\intop\hat{k}^{j}(\pi)(t)\cdot\hat{f}^{j}(t,\hat{y}^{j}(\sigma^{j}\otimes\pi)(t),u,v)\pi(t,u)(dv),
\]
for $\omega^{j}-$a.a. $\pi\in\mathcal{P}$ and $\tilde{\zeta}-$a.a
$(t,u)\in[t_{0},t_{1}]\times R,$ where $\tilde{\zeta}$ is the positive
Radon measure introduced in section \ref{sec:Relaxed-and-Hyper-relaxed}.

It follows that
\[
\intop\omega^{j}(d\pi)\intop[\mathfrak{h}^{j}(\pi,t,u)-\intop\hat{k}^{j}(\pi)(t)\cdot\hat{f}^{j}(t,\hat{y}^{j}(\sigma^{j}\otimes\pi)(t),u,v)\pi(t,u)(dv)]\tilde{\zeta}(dt,du)=0.
\]
Define the function 
\[
I^{j}(\pi):=\intop\left[\mathfrak{h}^{j}(\pi,t,u)-\intop\hat{k}^{j}(\pi)(t)\cdot\hat{f}^{j}(t,\hat{y}^{j}(\sigma^{j}\otimes\pi)(t),u,v)\pi(t,u)(dv)\right]\tilde{\zeta}(dt,du)
\]
which is $\omega^{j}-$integrable and the function
\[
I(\pi):=\int\left[\mathfrak{h}(\pi,t,u)-\intop\hat{k}(\pi)(t)\cdot\hat{f}(t,\hat{y}(\bar{\sigma}\otimes\pi)(t),u,v)\pi(t,u)(dv)\right]\tilde{\zeta}(dt,du)
\]
which is the pointwise limit of $I^{j}(\pi).$

With a similar analysis used above, we can add and subtract $\int I^{j}(\pi) \omega(d\pi)$ (still well defined by the continuity of $I^{j}(.)$) and 
$$ | \int I^{j}(\pi) (\omega-\omega^{j})(d\pi) | \leq K |\int (\omega-\omega^{j})(d\pi))| \rightarrow 0 $$
since $\omega^{j}\rightharpoondown \omega$ weakly-{*}, where $K:=2L_{\hat{h}} ||\chi(.)||_{L^{1}}$.
Then, with the help of the dominated convergence theorem, it follows that
\[
\intop\omega(d\pi)\intop\left[\mathfrak{h}(\pi,t,u)-\intop\hat{k}(\pi)(t)\cdot\hat{f}(t,\hat{y}(\bar{\sigma}\otimes\pi)(t),u,v)\pi(t,u)(dv)\right]\tilde{\zeta}(dt,du)=0.
\]
From the definition of $\mathfrak{h}^{j}(.,.,.)$, it follows that
$I^{j}(\pi)\geq0$ $\omega^{j}-$a.a. $\pi\in\mathcal{P},$ as well
as its limit function $I(\pi)\geq 0$
 $\omega-$a.a. $\pi\in\mathcal{P}.$ Using lemma \ref{lem:X.2.2 Warga},
we obtain

\[
\intop\omega(d\pi)\intop_{t_{0}}^{t_{1}}dt\intop\left[\mathfrak{h}(\pi,t,u)-\intop\hat{k}(\pi)(t)\cdot\hat{f}(t,\hat{y}(\bar{\sigma}\otimes\pi)(t),u,v)\pi(t,u)(dv)\right]\bar{\sigma}(t)(du)=0.
\]
 Since all the measures involved in the integrals are positive definite,
it follows that 
\[
\mathfrak{h}(\pi,t,u)=\intop\hat{k}(\pi)(t)\cdot\hat{f}(t,\hat{y}(\bar{\sigma}\otimes\pi)(t),u,v)\pi(t,u)(dv),
\]
$\omega-$a.a. $\pi\in\mathcal{P},$ a.e. $t\in[t_{0},t_{1}],$ $\bar{\sigma}(t)-$a.a.
$r\in R.$

\textbf{Step 3: }We now derive relation $iv)$ from (\cite{key-4},
Theorem X.3.5, point (2)). Define the functions
\[
k^{j}(t):=\left(l_{0}^{j}\partial_{x}h_{0}^{j}(y^{j}(\sigma^{j})(t_{1}))+l_{1}^{j}\partial_{x}h_{1}^{j}(y^{j}(\sigma^{j})(t_{1}))\right)Z^{j}(t)
\]
and 

\[
H^{j}(t,u):=k^{j}(t)\cdot f^{j}(t,y^{j}(\sigma^{j})(t),u)+\intop\mathfrak{h}^{j}(\pi,t,u)\omega^{j}(d\pi).
\]
From (\cite{key-4}, Theorem X.3.5, point (2)), it follows that
\[
\intop H^{j}(t,u)(\sigma-\sigma^{j})(t)(du)\geq0,\qquad\mathrm{a.e.}\; t\in[t_{0},t_{1}],
\]
which can be explicitly written as
\[
\intop k^{j}(t)\cdot f^{j}(t,y^{j}(\sigma^{j})(t),u)(\sigma-\sigma^{j})(t)(du)+\intop(\sigma-\sigma^{j})(t)(du)\intop\mathfrak{h}^{j}(\pi,t,u)\omega^{j}(d\pi)\geq0,
\]
for every $\sigma\in\mathrm{co}\mathcal{S}'$, a.e. $t\in[t_{0},t_{1}].$
We now observe that the function $r\mapsto\int\mathfrak{h}^{j}(\pi,t,u)\omega^{j}(d\pi)$
is continuous for every $j\in\mathbb{N},$ a.e. $t\in[t_{0},t_{1}].$ Furthermore, adding and subtracting the term $\int\mathfrak{h}^{j}(\pi,t,u)\omega(d\pi)$, we can prove that $r\mapsto\int\mathfrak{h}^{j}(\pi,t,u)\omega^{j}(d\pi)$ converges uniformly to $r\mapsto\int\mathfrak{h}(\pi,t,u)\omega(d\pi)$, a.e. $t\in[t_{0},t_{1}]$ (in turns, we have used the weakly-* convergence of the sequence $\{ \omega^{j} \}$ and the continuity of the the function $r\mapsto \mathfrak{h}(\pi,r,t)$, for every $\pi\in \mathcal{P}$ and a.e. $t\in [t_{0},t_{1}]$). This in particular implies that the functions $r\mapsto\int\mathfrak{h}(\pi,t,u)\omega(d\pi)$ and $r\mapsto H(t,u)$  are continuous, a.e. $t\in [t_{0}, t_{1}]$.
A use of  the dominated convergence theorem and the convergence of $\sigma^{j} \rightharpoondown \bar{\sigma}$ weakly-*
permits us to pass to the limit in the relation above, yielding
\[
\intop k(t)\cdot f(t,y(\bar{\sigma})(t),u)(\sigma-\bar{\sigma})(t)(du)+\intop(\sigma-\bar{\sigma})(t)(du)\intop\mathfrak{h}(\pi,t,u)\omega(d\pi)\geq0,
\]
which is exactly the relation
\[
\intop H(t,u)(\sigma-\bar{\sigma})(t)(du)\geq0
\]
for every $\sigma\in\mathrm{co}\mathcal{S}'$, a.e. $t\in[t_{0},t_{1}].$
By integrating with respect to $t$ on $[t_{0},t_{1}],$ we obtain
relation $iv).$

\textbf{Step 4: }In this last step, we derive a pointwise condition
for the function $H(.,.).$ We preliminary observe that the $H(.,.)$
is integrable with respect to $t$ for every $r\in R$ and continuous
with respect to $r$ a.e. $t\in[t_{0},t_{1}],$ since the functions
$(t,u)\mapsto\mathfrak{h}(\pi,t,u)$ and $(t,u)\mapsto k(t)f(t,y(\bar{\sigma})(t),u)$
satisfy the same property. 

We recall that $\mathcal{S}'$ is defined as in section \ref{sec:Choice Den Set}.
It follows that, for every $u(.)\in\mathcal{R}_{\infty},$ there exists
a null set $Z_{u}\subset[t_{0},t_{1}]$ such that
\[
H(t,u(t))-\intop H(t,u)\bar{\sigma}(t)(du)\geq0\qquad\forall t\,\in[t_{0},t_{1}]\backslash Z_{u}.
\]
Since the set $\mathcal{R}_{\infty}$ is denumerable, then $Z:=\cup_{u\in\mathcal{R}_{\infty}}Z_{u}$
is still a null set. The set $\{u(t):\, u\in\mathcal{R}_{\infty}\}$
is dense in $R(t)$ a.e. $t\in[t_{0},t_{1}]$ and this implies
\[
\intop H(t,u)\bar{\sigma}(t)=\inf_{u\in\mathcal{R}_{\infty}}H(t,u(t))=\min_{u\in U(t)}H(t,u).
\]
This completes the proof.
\end{proof}
The result proved for hyperrelaxed controls can be similarly derived
also for relaxed adverse control problems. In this case, we do not
need to choose any denumerable subset $\mathcal{S}'$ of $\mathcal{S}$
and the analysis convergence is more straightforward. Now we denote
by $y^{j}(\sigma)(.)$ the unique solution of equation (\ref{eq:Pert_Eq_1})
for $\sigma\in\mathcal{S}$ and by $y^{j}(\sigma\times\sigma_{P})(.)$
the unique solution of (\ref{eq:Perturb_Eq_2}) for $\sigma\in\mathcal{S}$
and $\sigma_{P}\in\mathcal{S}_{P}$. The problem $(RP^{j})$ has a solution
$\sigma^{j}\in\mathcal{S}$ for every $j.$ The function $Z^{j}(t)$
is the same of theorem \ref{thm:Hyper-Relaxed}, while
\[
\hat{Z}^{j}(\sigma_{P})(t)=I_{n+m}+\intop_{t}^{t_{1}}\hat{Z}^{j}(\sigma_{P})(s)\partial_{x}\hat{f}^{j}(s,\hat{y}^{j}(\sigma^{j}\times\sigma_{P})(s),(\sigma^{j}\times\sigma_{P})(s))ds
\]
We deal with the derivatives of the functions $h_{1}^{j}(.)$
and $\hat{h}^{j}(.)$ instead that considering the derivatives of $H_{1}^{j}(.)$ and
$\hat{H}^{j}(.)$. In what follows, we just sketch the proof in the relaxed case, pointing out the main differences with analysis carried out in theorem \ref{thm:Hyper-Relaxed}.
\begin{thm}
\label{thm:Relaxed Prob}Let $(\bar{\sigma},\bar{b}, \bar{b}_{P})$ be an optimal
solution to the problem $(RP)$. Then there exist a set of index $J\subset\mathbb{N}$
, limiting multipliers $l_{0}\geq0,$ $l_{1}\in\mathbb{R},$ limiting
initial directions $\mathcal{H}_{0},\mathcal{H}_{1}\in\mathbb{R}^{n},$
a $\omega\in\mathrm{f.r.m.^{+}}(\mathcal{S}_{P})$, a $\omega-$integrable
function $\hat{\mathcal{H}}:\mathcal{S}_{P}\rightarrow\mathbb{R}^{n+m},$
and  the continuous functions
$Z:[t_{0},t_{1}]\rightarrow M_{n\times n},$ $\hat{Z}:[t_{0},t_{1}]\times\mathcal{S}_{P}\rightarrow M_{\left(n+m\right)\times(n+m)},$
such that:

\[
i)\quad Z(t)=\lim_{j\in J}Z^{j}(t), \qquad  \hat{Z}(\sigma_{P})(t)=\lim_{j\in J}\;\hat{Z}^{j}(\sigma_{P})(t)  \quad \mathrm{uniformly} \; t\in[t_{0},t_{1}], \; \sigma_{P}\in\mathcal{S}_{P},
\]

\[
(l_{0},l_{1})=\lim_{j\in J}(l_{0}^{j},l_{1}^{j}),\quad\mathcal{H}_{0}=\lim_{j\in J}\partial_{x}h_{0}^{j}(y^{j}(t_{1})),\quad\mathcal{H}_{1}=\lim_{j\in J}\partial_{x}h_{1}^{j}(y^{j}(t_{1})),
\]
\[
\hat{\mathcal{H}}(\sigma_{P})=\lim_{j\in J}\partial_{x}\hat{h}^{j}(\hat{y}(\sigma^{j}\times\sigma_{P})(t_{1})),\qquad\omega-a.a.\quad\sigma_{P}\in\mathcal{S}_{P};
\]

\[
ii)\qquad\quad0<l_{0}+\left|l_{1}\right|+\omega(\mathcal{S}_{P})\leq1
\]

Define:
\[
k(t):=\left(l_{0}\mathcal{H}_{0}+l_{1}\mathcal{H}_{1}\right)Z(t),\quad\hat{k}(\sigma_{P})(t):=\hat{\mathcal{H}}(\sigma_{P})\hat{Z}(\sigma_{P})(t),
\]
\[
\mathfrak{h}(\sigma_{P},t,s):=\max_{v\in v(t)}\hat{k}(\sigma_{P})(t)\cdot\intop\hat{f}(t,\hat{y}(\bar{\sigma}\times\sigma_{P})(t),u,v)s(du),
\]

\[
H(t,s):=k(t)f(t,y(\bar{\sigma})(t),s)+\intop\mathfrak{h}(\sigma_{P},t,s)\omega(d\sigma_{P}),\quad(s\in\mathcal{S}).
\]
Then
\[
iii)\qquad\;\intop H(t,u)\bar{\sigma}(t)(du)=\min_{u\in U(t)}H(t,u)\quad a.e.\; t\in[t_{0},t_{1}],
\]

\[
iv)\qquad\hat{h}(\hat{y}(\bar{\sigma}\times\sigma_{P}^{*})(t_{1})=\max_{\sigma_{P}\in\mathcal{S}_{P}}\hat{h}(\hat{y}(\bar{\sigma}\times\sigma_{P})(t_{1}))\quad\omega-a.a.\:\sigma_{P}^{*}\in\mathcal{S}_{P}.
\]
\end{thm}
\begin{proof}
\textbf{Step 1: } The first part of the proof retraces the step 1 of
theorem \ref{thm:Hyper-Relaxed} and here we omit the details. We
just recall that, if $\sigma^{j}\rightharpoondown\sigma$ weakly-{*}
in $\mathcal{S},$ then also $\sigma^{j}\times\sigma_{P}\rightharpoondown\sigma\times\sigma_{P}$
weakly-{*} in $\mathcal{Q}.$ Furthermore the mapping $\sigma_{P}\mapsto\sigma\times\sigma_{P}$
is continuous as well as the function $\sigma_{P}\mapsto\hat{y}^{j}(\sigma\times\sigma_{P})$
for every $j.$ This remark justifies the use of theorem \ref{thm:Hyper-Relaxed},
step 1 arguments. The perturbed problems $(RP^{j})$ have $C^{1}$
data and solutions $\sigma^{j}\in\mathcal{S}$. It follows that theorem
(\cite{key-4}, X.3.7, pp 512) can be applied. In particular from
(\cite{key-4}, X.3.5, pp 505, point (1)), it follows that there exist
$l_{0}^{j}\geq0,$ $l_{1}^{j}\in\mathbb{R}^{n}$ and $\omega^{j}\in\mathrm{f.r.m.}^{+}(\mathcal{S}_{P})$
such that, 
\[
0<l_{0}^{j}+|l_{1}^{j}|+\omega^{j}(\mathcal{S}_{P})\leq1
\]
and, considering a subsequence of index $J_{2}\subset J_{1},$ we
obtain $l_{0}\geq0,$ $l_{1}\in\mathbb{R}^{n}$ and $\omega\in\mathrm{f.r.m.}^{+}(\mathcal{S}_{P})$
such that
\[
\lim_{j\in J_{2}}l_{0}^{j}=l_{0},\qquad\lim_{j\in J_{2}}l_{1}^{j}=l_{1},\qquad\lim_{j\in J_{2}}\omega^{j}=\omega\quad\mathrm{weakly-*}
\]
and such that
\[
0<l_{0}+|l_{1}|+\omega(\mathcal{S}_{P})\leq1.
\]

From lemma \ref{lem:Properties_Perturbed_Prob} $(i)$ and using the
result of (\cite{key-4}, Theorem VI.I.6, pp. 348), there exists $J_{3}\subset J_{2}$
such that
\[
\lim_{j\in J_{3}}y^{j}(\sigma^{j})(t)=\lim_{j\in J_{3}}y(\sigma^{j})(t)=y(\bar{\sigma})(t)
\]
and, since $\sigma^{j}\times\sigma_{P}\rightharpoondown\bar{\sigma}\times\sigma_{P}$
in $\mathcal{Q}$, we have 
\[
\lim_{j\in J_{3}}\hat{y}^{j}(\sigma^{j}\times\sigma_{P})(t)=\lim_{j\in J_{3}}\hat{y}(\sigma^{j}\times\sigma_{P})(t)=\hat{y}(\bar{\sigma}\times\sigma_{P})(t)
\]
uniformly with respect to $t\in[t_{0},t_{1}],$ $\sigma_{P}\in \mathcal{S}_{P}.$ (Again, the fact that $\sigma^{j}\rightharpoondown \bar{\sigma}$ follows from the optimality of $\sigma^{j}$. Indeed
$$ \lim_{j\in J_{3}} h^{j}_{0}(y^{j}(\sigma^{j})(t_{1})=h_{0}(y(\bar\sigma)(t_{1})) \leq \lim_{j\in J_{3}} h^{j}_{0}(y^{j}(\sigma)(t_{1})=h_{0}(y(\sigma)(t_{1}))$$
for all $\sigma \in \mathcal{S}$).

By the same analysis described in theorem \ref{thm:Hyper-Relaxed},
there exist $J\subset J_{3},$ $\mathcal{H}_{0},\mathcal{H}_{1}\in\mathbb{R}^{n}$
such that 
\[
\lim_{j\in J}\partial_{x}h_{0}^{j}(y^{j}(\sigma^{j})(t_{1}))=\mathcal{H}_{0},\qquad\lim_{j\in J}\partial_{x}h_{1}^{j}(y^{j}(\sigma^{j})(t_{1}))=\mathcal{H}_{1},
\]
and, for each fixed $\sigma_{P}\in\mathcal{S}_{P},$ a vector $\hat{\mathcal{H}}(\sigma_{P})$
such that 
\[
\lim_{j\in J}\partial_{x}\hat{h}(\hat{y}^{j}(\sigma^{j}\times\sigma_{P})(t_{1}))=\hat{\mathcal{H}}(\sigma_{P}).
\]
Relations $i)$ and $ii)$ are proved. Since the function $\hat{\mathcal{H}}:\mathcal{P}\rightarrow \mathbb{R}^{n}$ is the pointwise limit of the sequence of continuous functions $\{ \sigma_{P} \mapsto\partial_{x}\hat{h}(\hat{y}^{j}(\sigma^{j}\times\sigma_{P})(t_{1})) \}_{j\in J}$, it follows that $\hat{\mathcal{H}}(.)$ is $\mathcal{B}(\mathcal{S}_{P})-$measurable.

\textbf{Step 2}: From (\cite{key-4}, Theorem X.3.7, pp. 512, point
$(3)$), we have 
\[
\hat{h}^{j}(\hat{y}^{j}(\sigma^{j}\times\sigma_{P}^{*})(t_{1}))=\max_{\sigma_{P}\in\mathcal{S}_{P}}\hat{h}^{j}(\hat{y}^{j}(\sigma^{j}\times\sigma_{P})(t_{1}))\qquad\omega^{j}-a.a.\:\sigma_{P}^{*}\in\mathcal{S}_{P}.
\]
Expressing the relation above in integral form and using the same convergence analysis showed in theorem \ref{thm:Hyper-Relaxed}, Step 2, we obtain
\[
\hat{h}(\hat{y}(\bar{\sigma}\times\sigma_{P}^{*})(t_{1})=\max_{\sigma_{P}\in\mathcal{S}_{P}}\hat{h}(\hat{y}(\bar{\sigma}\times\sigma_{P})(t_{1}))\quad\omega-a.a.\:\sigma_{P}^{*}\in\mathcal{S}_{P}.
\]
We set 
\[
\hat{k}^{j}(\sigma_{P})(t):=\partial_{x}\hat{h}^{j}(\hat{y}^{j}(\sigma^{j}\times\sigma_{P})(t_{1}))\hat{Z}^{j}(\sigma_{P})(t),
\]
and
\[
\mathfrak{h}^{j}(\sigma_{P},t,s):=\max_{v\in v}\hat{k}^{j}(\sigma_{P})(t)\cdot\intop\hat{f}^{j}(t,\hat{y}^{j}(\sigma^{j}\times\sigma_{P})(t),u,v)s(du),\quad(s\in\mathcal{S}).
\]
From (\cite{key-4}, Theorem X.3.7, pp. 512, point $(4)$ ), it follows
that  
\[
\mathfrak{h}^{j}(\sigma_{P},t,\sigma^{j}(t))=\intop\hat{k}^{j}(\sigma_{P})(t)\hat{f}^{j}(t,\hat{y}^{j}(\sigma^{j}\times\sigma_{P})(t),\sigma^{j}(t),v)\sigma_{P}(t)(dv),
\]
for $\omega^{j}-$a.a. $\sigma_{P}\in\mathcal{S}_{P},$ a.e. $t\in[t_{0},t_{1}].$  Preliminary, we observe that the function $(\sigma_{P},t)\rightarrow\mathfrak{h}^{j}(\sigma_{P},t,\sigma(t))$ is $\omega^{j}\times dt-$integrable for every $\sigma\in\mathcal{S}.$ We can write the relation above as
\[
\intop\omega^{j}(d\sigma_{P})\intop_{t_{0}}^{t_{1}}[\mathfrak{h}^{j}(\sigma_{P},t,\sigma^{j}(t))-
\]
\[
-\intop\hat{k}^{j}(\sigma_{P})(t)\hat{f}^{j}(t,\hat{y}^{j}(\sigma^{j}\times\sigma_{P})(t),\sigma^{j}(t),v)\sigma_{P}(t)(dv)]dt=0
\]
and, using the same procedure carried out in theorem \ref{thm:Hyper-Relaxed}, Step 2, where this time we deal with the functions $$I^{j}(\sigma_{P}):= \intop_{t_{0}}^{t_{1}}[\mathfrak{h}^{j}(\sigma_{P},t,\sigma^{j}(t))-\intop\hat{k}^{j}(\sigma_{P})(t)\hat{f}^{j}(t,\hat{y}^{j}(\sigma^{j}\times\sigma_{P})(t),\sigma^{j}(t),v)\sigma_{P}(t)(dv)]dt$$
and the pointwise limit
$$I(\sigma_{P}):= \intop_{t_{0}}^{t_{1}}[\mathfrak{h}(\sigma_{P},t,\bar{\sigma}(t))-\intop\hat{k}(\sigma_{P})(t)\hat{f}(t,\hat{y}(\bar{\sigma} \times\sigma_{P})(t),\bar{\sigma}(t),v)\sigma_{P}(t)(dv)]dt,$$
it follows that

\[
\mathfrak{h}(\sigma_{P},t,\bar{\sigma}(t))=\intop\hat{k}(\sigma_{P})(t)\hat{f}(t,\hat{y}(\bar{\sigma}\times\sigma_{P})(t),\bar{\sigma}(t),v)\sigma_{P}(t)(dv),
\]
 $\omega-$a.a. $\sigma_{P}\in\mathcal{S}_{P},$ a.e. $t\in[t_{0},t_{1}].$
We observe that, in this case, we do not use $\tilde{\zeta}$
in the convergence analysis.

\textbf{Step 3}: We follow the same approach used in theorem \ref{thm:Hyper-Relaxed}, steps 3-4.\\
Define the functions
\[
k^{j}(t):=\left(l_{0}^{j}\partial_{x}h_{0}^{j}(y^{j}(\sigma^{j})(t_{1}))+l_{1}^{j}\partial_{x}h_{1}^{j}(y^{j}(\sigma^{j})(t_{1}))\right)Z^{j}(t)
\]
and
\[
H^{j}(t,s):=k^{j}(t)\cdot f^{j}(t,y(\sigma^{j})(t),s)+\intop \mathfrak{h}^{j}(\sigma_{P}, t, s)\omega^{j}(\sigma_{P}),
\]
for every $s\in \mathcal{S}$, a.e. $t\in [t_{0}, t_{1}]$. (Notice that this time $s$ is a relaxed control).
It is easy to check that the sequence of functions $\{t \mapsto k^{j}(t)\}_{j\in \mathbb{N}}$ is continuous and $\{t \mapsto H^{j}(t,s)\}_{j\in \mathbb{N}}$ is integrable for every $s\in \mathcal{S}$. 
From (\cite{key-4}, Theorem X.3.7, point (2)), it follows that
\[
\intop H^{j}(t,u)(s-\sigma^{j})(t)(du)\geq0,
\]
for every $s\in \mathcal{S}$ and a.e. $ t\in[t_{0},t_{1}].$
Recalling that $\sigma^{j}\rightharpoondown \bar{\sigma}$ weakly-*, by a similar (and simpler) convergence analysis described in theorem \ref{thm:Hyper-Relaxed}, we obtain
\[
\intop H(t,u)(s-\bar{\sigma})(t)(du)\geq0
\]
for every $s \in \mathcal{S}$, a.e. $t\in [t_{0},t_{1}]$.\\
From this inequality we can easily
derive the maximum principle $iii)$ integrating with respect the time $t$ over $[t_{0}, t_{1}]$ and using the denumerable family of controls/rational endpoints of $[t_{0}, t_{1}]$ defined
in section \ref{sec:Choice Den Set}. This completes the proof.
\end{proof}

\textbf{Comments:}
\begin{enumerate}

\item Condition (\cite{key-4}, Theorems X.3.5, X.3.7, point (3)) is established for problems described by the implicit constraint 
$$ \hat{\varphi}(\pi):=\varphi(\sigma,\pi) \in A,\qquad \forall \,\pi \in \mathcal{P}$$
where $\sigma\in \mathcal{S}$ is fixed, $A$ is a convex subset of $\mathbb{R}^{n+m}$ with nonempty interior and $\varphi: \mathcal{S}\times \mathcal{P}\rightarrow \mathbb{R}^{n+m}$ is a continuous function. Our formulation of the problem is slightly different. However, we can set $A=(-\infty,0]$, $\pi \mapsto \hat{\varphi}(\pi):=\hat{H}^{j}(\hat{y}(\sigma\otimes\pi)(t_{1})$ (see section \ref{sec:Perturbed-Problem})
and interpret $\omega^{j}$ as the dual variable $l^{j}\in C(\mathcal{P})^{*}$ such that
$$ l^{j}(\hat{H}^{j}):=\intop \hat{H}^{j}(\hat{y}(\sigma\otimes\pi)(t_{1})\omega^{j}(\pi)=\max \{l^{j}(g):\, g\in C(\mathcal{P}),\,g(\pi)\leq0 \; \forall \, \pi\in \mathcal{P} \}.$$ 
Since the functions $\pi \mapsto \hat{H}^{j}(\hat{y}(\sigma\otimes\pi)(t_{1})$ are non positive and bounded, $\omega^{j} \in \mathrm{f.r.m.}^{+}(\mathcal{P})$ and relation
\[
\hat{H}^{j}(\hat{y}(\sigma\otimes\pi^{*})(t_{1})=\max_{\pi\in\mathcal{P}}\hat{H}^{j}(\hat{y}(\sigma\otimes\pi)(t_{1}),\quad\omega^{j}-\mathrm{a.a.}\,\pi^{*}
\]
formally follows.

\item  Theorems X.3.5, X.3.7 in \cite{key-4} provide also conditions for the optimal initial points $(\bar{b},\bar{b}_{P})$. Here, to sake of clarity, we have not considered these conditions and we have chosen perturbed problems with fixed initial states. However, the analysis could be easily adapted by taking into account perturbed problems with respect to also the initial state condition.

\item The notion of ``relaxed derivative" appears in the adjoint  equations $Z(.)$,  $\hat{Z}(.)(.)$ and the related functions of the theorems \ref{thm:Hyper-Relaxed}, \ref{thm:Relaxed Prob}. The convergence analysis used to obtain such functions avoids some measurability issues which come out dealing with derivative containers (see \cite{key-5}, comments in section 3 and following discussion). In the present theorems, the measurability is guaranteed by the limit process.

\item The use of the function $\mathfrak{h}(.,.,.)$ makes a breakthrough with respect to other necessary conditions for nonsmooth problems obtained in the earlier literature. We believe that theoretical and computational methods based on the present necessary conditions might take advantage of the ``maximization" related to function $\mathfrak{h}(.,.,.)$ and of the ``minimization" related to function $H(.,.)$ (or, to be more precise, the maximization and minimization process related to the sequence of functions $\mathfrak{h}^{j}(.,.,.)$ and $H^{j}(.,.)$, respectively). This will be matter of studies in following papers.

\end{enumerate}

\section{An Example}

Consider the minimax optimal control problem
\[
\left\{ \begin{array}{l}
\mathrm{Minimize}_{u\in \mathcal{U}}\max_{v\in \mathcal{V}}y(u,v)(1)\\
\mathrm{over\: measurable\: functions\:}u(.), v(.)\mathrm{\: such\: that}\\
u(t)\in \{-1,1\},\quad v(t)\in \{-1,1\}\qquad\mathrm{a.e.}\; t\in[0,1]\\
\dot{y}(t)=|y(t)|u(t)v(t)\quad\mathrm{a.e.}\: t\in[0,1]\\
y(0)=1\quad
\end{array}\right. .
\]
This problem can be reformulated as
\[
(E)\,\left\{ \begin{array}{l}
\mathrm{Minimize}_{u\in \mathcal{R}}\; \alpha \\
\mathrm{over\: measurable\: functions\:}u(.), v(.)\mathrm{\: such\: that}\\
u(t)\in \{-1,1\},\quad v(t)\in \{-1,1\}\qquad\mathrm{a.e.}\; t\in[0,1]\\
\dot{\alpha}(t)=0,\\
\dot{y}(t)=|y(t)|u(t)v(t)\quad\mathrm{a.e.}\: t\in[0,1]\\
y(0)=1\quad\\
y(u,v)(1) - \alpha \leq0 \qquad \forall \, v\in \mathcal{V}
\end{array}\right. ,
\] 
We aim to study the hyer-relaxed version of problem$(E)$, which concerns a problem with same data, but where the dynamic constraint is expressed by
\[
y(\sigma\otimes \pi)(t)=1+\int_{0}^{t}ds\int |y(s)|\,  \sigma(s)(du) \int r\, v \pi(s,u)(dv)
\]
and the relation $y(u,v)(1) - \alpha \leq0$ becomes
\[
y(\sigma \otimes \pi)(1) - \alpha \leq0 \qquad \forall \, \pi\in \mathcal{P}.
\]

From Theorem \ref{thm:Hyper-Relaxed}, condition $(i)$, it follows that there exists a sequence $\hat{Z}^{\varepsilon}(\pi)(t)\rightarrow \hat{Z}(\pi)(t)$ converging uniformly w.r.t. $\pi\in\mathcal{P}$ and $t\in[0,1]$ to some function $\hat{Z}(\pi)(t)$. The function $\mathfrak{h}(\pi,t,u)$ has the form

$$\mathfrak{h}(\pi,t,u)=\max_{v\in \{-1,1\}} \hat{Z}(\pi)(t)|y|r\, v=|y||\hat{Z}(\pi)(t)|$$

and does not depend on $r$. This implies that condition $v)$ of Theorem \ref{thm:Hyper-Relaxed} is satisfied for every  $u\in \{ -1, 1 \}$. So we can choose an arbitrary control $\bar{\sigma}(t)=\delta_{u(t)}$ with, for instance, $u(t)\equiv 1$. Plugging such a control into the hyper-relaxed dynamics and looking at the function $y(\sigma \otimes \pi)(1) $, we observe that it is maximum when $\pi(t,u)(dv)=\delta_{u(t)}(dv)$ and that optimal solution is given by the solution of the ordinary differential equation $\dot{y}(t)=|y(t)|$, $y(0)=1$.

\noindent
\section*{Acknowledgments}

 This work was co-funded by the European Union under the 7th Framework Program �FP7-PEOPLE-2010-ITN�, grant agreement number 264735-SADCO. I thank Prof. R. B. Vinter for having brought to my attention many papers on the topic and the anonymous referees for their helpful comments. 

%






\end{document}